\documentclass[11pt]{amsart}
\title{Algebraic $K$-theory for squares categories}
\author[Campbell]{Jonathan Campbell}
\address{Center for Communications Research - La Jolla}
\email{jonalfcam@gmail.com}
\author[Kuijper]{Josefien Kuijper}
\address{Department of Mathematics, University of Toronto}
\email{josefien.kuijper@utoronto.ca}
\author[Merling]{Mona Merling}
\address{Department of Mathematics, University of Pennsylvania}
\email{mmerling@math.upenn.edu}
\author[Zakharevich]{Inna Zakharevich}
\address{Department of Mathematics, Cornell University}
\email{zakh@math.cornell.edu}

\usepackage{CZ, color}
\usepackage{tikz-cd}
\usepackage{mathtools}
\usepackage[normalem]{ulem}
\newcommand{\sbt}{\,\begin{picture}(-1,1)(0.5,-1)\circle*{1.8}\end{picture}\hspace{.05cm}}
\newcommand{\sO}{\mathcal{O}}
\newcommand{\B}{\mathcal{B}}
\newcommand{\Span}{\mathbf{Span}}
\fullpage

\newcommand{\edit}[1]{{\color{orange}{#1}} }
\renewcommand{\edit}[1]{{\color{black}{#1}} }
\renewcommand{\sout}[1]{}

\usepackage{preamble_refs}

\newlength{\storeparskip}
\begin{document}

\date{\today}

\begin{abstract}
In this paper we introduce a new formalism for $K$-theory, called \emph{squares $K$-theory}.  This formalism allows us to simultaneously generalize the usual three-term relation $[B] = [A] + [C]$ for an exact sequence $A \hookrightarrow B \twoheadrightarrow C$ or 
 for a subtractive sequence $A\hookrightarrow B \leftarrow C$, by defining $K_0$ of \edit{a squares category\sout{exact and subtractive categories}} to satisfy a four-term relation $[A]+[D]= [C] + [B]$ for a ``good'' square diagram with these corners. 
Examples that rely on this formalism are $K$-theory of smooth manifolds of a fixed dimension and $K$-theory of (smooth and) complete varieties.
Another application we give of this theory is  the construction of a derived motivic measure taking value in the $K$-theory of homotopy sheaves.
\end{abstract}
\maketitle

\begingroup%
\setcounter{secnumdepth}{1}
\setcounter{tocdepth}{1}
\tableofcontents
\endgroup%

\section*{Introduction}

Higher algebraic $K$-theory was originally developed by Quillen in \cite{quillen} in order to construct groups $K_i(R)$ for commutative rings.  Quillen associated to any exact category $\C$ a space $\Omega BQ\C$, and defined the groups $K_i(\C)$ to be the homotopy groups of this space.  For a ring $R$, $K_i(R)$ is simply $K_i(\mathbf{Proj}_R)$, the $K$-groups of the category of finitely generated projective $R$-modules.  Although these higher groups are generally difficult to compute (see \cite{grayson_binary} for a description involving generators and relations) the group $K_0(\C)$ is simple: it is the free abelian group generated by objects of $\C$, modulo the relation 
\[[B] = [A] + [C] \qquad \hbox{if there exists an exact sequence}\   A \rightarrow B \rightarrow C.\]
The groups $K_i(\C)$ are the ``higher'' analogs of these groups, encapsulating higher homotopical information about how objects decompose and reassemble. A variant of higher algebraic $K$-theory for ``subtractive categories", which instead encodes relations
\[[B] = [A] + [C] \qquad \hbox{if there exists a ``subtractive sequence"}\   A \rightarrow B \leftarrow C\]
was introduced in \cite{campbell} in order to give a new construction of the $K$-theory of varieties from \cite{Z-Kth-ass}.

Although three-term relations appears often, there are situations, especially in geometry, where they are insufficient.  Consider, for example, the following four-term relation on polytopes:
\[[P\cup Q] + [P \cap Q] = [P] + [Q].\]
This relation appears in McMullen's polytope algebra \cite{mcmullen}, in which decomposition and assembly of closed polytopes of mixed dimensions is considered.\footnote{This should be contrasted with the classical scissors congruence groups of polytopes, where a particular dimension of polytope is fixed and all lower-dimensional intersections are disregarded 
 \cite{dupont, sah79}. 
 This type of decomposition works well with the classic three-term relation.}   It also appears in Bittner's presentation of the Grothendieck ring of varieties \cite{bittner04}.
In such circumstances it is not possible to use a three-term relation without extensive artificial constructions (see \cite[Section 5.2]{Z-Kth-ass} for a discussion).  It is thus desirable to construct a version of $K$-theory in which four-term relations appear naturally.
 
 Another example comes from considering scissors congruence of manifolds, also called $SK$ (In German, ``schneiden und kleben," cut and paste) equivalence. Given a closed smooth oriented manifold $M$, one can cut it along a separating codimension $1$ submanifold $\Sigma$ with trivial normal bundle, and paste back the two pieces along an orientation preserving diffeomorphism $\Sigma\rightarrow \Sigma$ to obtain a new manifold, which we say is ``cut and paste equivalent" or ``scissors congruent" to $M$ \cite{KKNO}. In order to categorify the resulting scissors congruence groups of manifolds, it is crucial to consider four term relations.

In \cite{CZ-cgw} the notion of CGW-categories, which encode a certain type of four-term relation, is considered.  However, these are insufficient, as they still assume a fundamental ``three-terminess'' to the structure at hand: all morphisms are assumed to have either a kernel or a cokernel, and thus appear in an ``exact sequence'' of some type.  However, in all of the examples considered above, the problem is exactly that there are situations in which morphisms have neither a kernel nor a cokernel.  To understand how this works, consider the case of polytopes of a fixed dimension.  Given an isometric embedding of closed polytopes $P \rcofib Q$, the inclusion $\overline{Q \backslash P} \rcofib Q$ has a measure-$0$ intersection with $P$.  Since we are disregarding lower dimensional intersections, we can declare that
\[P \rcofib Q \lcofib \overline{Q \backslash P}\]
is an ``exact sequence'' and work with it analogously to the way that Quillen's construction works for the construction of $\Omega B Q\C$.  (\cite{CZ-cgw} makes this perspective rigorous.)

However, if we are considering closed polytopes of \emph{all} dimensions, this is no longer the case: $P \cap \overline{Q \backslash P} \neq \emptyset$, and so 
\[P \rcofib Q \lcofib \overline{Q \backslash P}\]
cannot be considered an exact sequence.  To understand how to fix this, we need the following observation about abelian categories: in an abelian category, the sequence
\[A \rcofib^{f+f'} B\oplus C \rfib^{g-g'} D\]
is exact if and only if the square 
\begin{diagram}
    {A & B \\ C & D\\};
    \arrowsquare{f}{f'}{g'}{g}
\end{diagram}
is \emph{stable}, i.e., is both a pushout and a pullback. The relation corresponding to this square is 
\[[A] + [D] = [B] + [C].\]
This perspective is precisely duplicated in the context of closed polytopes: the square
\begin{diagram} 
{ \overline{Q\backslash P} \cap P & \overline{Q \backslash P}  \\ P & Q \\};
\arrowsquare{}{}{}{}
\end{diagram}
is stable, and the relation given by this square is precisely the desired relationship between objects in McMullen's polytope algebra.  Similar phenomena occur with SK-invariants and with Bittner's presentation \edit{of the Grothendieck ring of varieties.}

The goal of this paper is to show how to construct the $K$-theory of categories with \emph{squares}.  This framing simultaneously generalizes the standard construction of the $K$-theory of \edit{any} Waldhausen category, of \edit{the subtractive Waldhausen (SW) category of varieties,} and of \edit{any} CGW-category, \edit{but it has the advantage that it} also works with the above geometric examples.  This approach has already produced new models of $K$-theories of manifolds which lift the scissors congruence (SK) groups, as investigated in \cite{WITMona, WIT2}. 

Squares $K$-theory also seems to be the right approach to study the $K$-theory spectrum of the category of varieties $K(\Var_k)$. We can define a squares category of complete varieties and show that $K(\Comp_k)$ is a retract of $K(\Var_k)$. Moreover we can define the $K$-theory spectrum of the category of smooth and complete varieties $K(\SmComp_k)$, raising the question of whether a spectrum-level analogue of Bittner's presentation \cite{bittner04} holds. In contrast, in the framework of assemblers, SW-categories or CGW-categories, the category of (smooth and) complete varieties is out of reach, since an open subvariety of a complete variety is not complete. As another application, we show that the definition of $K(\Var_k)$ using squares categories allows for the construction of a new derived motivic measure, taking values in the $K$-theory spectrum of the stable $\infty$-category of ``abstract compactly supported cohomology theories.'' (See Section~\ref{sec:derived}.) 

\subsection*{Acknowledgements}  The authors are grateful to Anna Marie Bohmann, Thomas Barnet-Lamb, Tony Elmendorf, Mike Mandell, Jack Morava, Dan Petersen, George Raptis, and Julia Semikina for enlightening discussions and insights regarding aspects of this paper.  We are also grateful to the anonymous referee for their helpful comments.

Two of the authors were partially supported by NSF grants for the Focused Research Group project Trace Methods and Applications for Cut-and-Paste $K$-theory, under the grant numbers DMS-2052988 (Merling) and DMS-2052977 (Zakharevich). Additionally, Merling was partially supported by NSF grant DMS-1943925 and Zakharevich was partially supported by NSF grant DMS-1846767. Kuijper was supported by ERC grant  ERC-2017-STG 759082.

\section{Squares categories}

In this section we introduce categories which have a notion of distinguished squares and a designated basepoint object, which we will call \emph{squares categories}. Then we define a $K$-theory construction for such categories, which breaks up the cut-and-paste data encoded by the squares.

\begin{definition}
    A \emph{simple double category} is a small double category where the $2$-cells are uniquely determined by their boundaries.  More rigorously, a simple double category $\C$ consists of:
    \begin{itemize}
        \item A pair of categories $(\E_\C,\M_\C)$ which have the same objects as $\C$. We denote morphisms in $\E_\C$ by $\rfib$, and morphisms in $\M_\C$ by \edit{$\rcofib$, when we want to emphasize in which of the two categories a morphism lives.}
        \item A subset of tuples of maps
        \[\mathscr{D} \subseteq \coprod_{A,B,C,D\in \ob\C} \Hom_{\M_\C}(A,B) \times \Hom_{\M_\C}(C,D) \times \Hom_{\E_\C}(A,C) \times \Hom_{\E_\C}(B,D).\]
        The tuples in $\mathscr{D}$ are the \emph{distinguished squares}, which are drawn
        \begin{diagram}
            { A & B \\ C & D. \\};
            \cofib{1-1}{1-2}^f \cofib{2-1}{2-2}^{f'}
            \fib{1-1}{2-1}_g \fib{1-2}{2-2}^{g'}
        \end{diagram}
    \end{itemize} 
    Distinguished squares compose both vertically and
    horizontally, in the sense that if each of the component squares is distinguished, so is the outside:
    \[\begin{inline-diagram} 
    {A & B \\ C & D \\ E & F \\};
    \fib{1-1}{2-1} \fib{2-1}{3-1}
    \fib{1-2}{2-2} \fib{2-2}{3-2}
    \cofib{1-1}{1-2} \cofib{2-1}{2-2} \cofib{3-1}{3-2}
    \end{inline-diagram} \qquad 
    \begin{inline-diagram}
    {A & B & C \\ D & E & F\\};
    \cofib{1-1}{1-2} \cofib{1-2}{1-3}
    \cofib{2-1}{2-2} \cofib{2-2}{2-3}
    \fib{1-1}{2-1} \fib{1-2}{2-2} \fib{1-3}{2-3}
    \end{inline-diagram}.\]
For any morphisms $f: A\mrto B$ and $g:B \rfib D$ the squares
 \[\begin{inline-diagram} 
    {A & B \\ A & B \\ };
    \eq{1-1}{2-1} 
    \eq{1-2}{2-2} 
    \cofib{1-1}{1-2}^{f} \cofib{2-1}{2-2}_{f}
    \end{inline-diagram}\  \mathrm{and} \ 
    \begin{inline-diagram}
    {A & A  \\ B & B \\};
    \eq{1-1}{1-2} 
    \eq{2-1}{2-2} 
    \fib{1-1}{2-1}_{g} \fib{1-2}{2-2}^{g} 
    \end{inline-diagram}\]
  must be distinguished.

  A functor $F:\C \rto \C'$ of simple double categories is a pair of functors $F_\E: \E_\C \rto \E_{\C'}$ and $F_\M: \M_\C \rto \M_{\C'}$ which agree on objects and preserve distinguished squares.  
\end{definition}

\begin{remark}
     The naming of $\E_\C$ and $\M_\C$ is meant to invoke ``epimorphisms'' and ``monomorphisms,'' and the notation reflects that intuition.  It is important to note, however, that in most examples the morphisms in $\E_\C$ are not, in fact, epic.
\end{remark}

\begin{definition} \label{def:catcofsq} 
A  
\textsl{squares category} is a simple double category $\C = (\E_\C,\M_\C)$ with a chosen basepoint $O$ which is initial in both $\M_C$ and $\E_C$.  
A functor of squares categories is a functor of simple double categories which preserves the basepoint.
\end{definition}


\begin{example} \label{ex:wald}
  Let $\mathcal{W}$ be a Waldhausen category. We can construct an associated squares category $\W^\square$ as follows: we set $\M_{\W^\square}$ be the category of cofibrations in $\mathcal{W}$, $\E_{\W^\square}$ to be the category $\mathcal{W}$, and we set the distinguished squares to be the commutative squares 
  \begin{diagram}
    { A & B \\ C & D \\};
    \cofib{1-1}{1-2} \cofib{2-1}{2-2} \to{1-1}{2-1} \to{1-2}{2-2}
  \end{diagram}
  in which the induced map $B \cup_A C \rto D$ is a weak equivalence. The base point $O$ is the zero object.
\end{example}

\begin{example} \label{ex:finset}
  We define the squares category of finite sets $\mathbf{FinSet}$, with $\E_\mathbf{FinSet}$ and $\M_\mathbf{FinSet}$ both the categories of finite sets and injections. Define a
  square
  \begin{diagram}
    { A & B \\ C & D \\};
    \arrowsquare{}{}{}{}
  \end{diagram}
  to be distinguished if $B \cap C = A$ and if $B \cup C = D$. The basepoint is the empty set.  (Here the
  intersection is inside their image in $D$, and the union is also inside $D$.)

  A similar construction works to construct a squares category for $\mathbf{Set}$; the finiteness is only relevant in order for the $K$-theory to be interesting.
\end{example}

 \begin{example}\label{ex:varcgw}   Let $\Var_k$ be the double category of varieties (i.e. reduced separated schemes \edit{of finite type)} over $k$ with $\M_{\Var_k}$ the \edit{\sout{closed}open} immersions and $\E_{\Var_k}$ the \edit{\sout{open}closed immersions}.  A square
 \begin{diagram}
     {A & B \\ C & D \\};
     \arrowsquare{}{}{}{}
 \end{diagram}
 is distinguished if it is a pullback square after forgetting down to the ordinary category of varieties and morphisms of varieties, and if the images of $B$ and $C$ cover $D$.

 \edit{\sout{Let $\Var_k^i$ be the squares subcategory of $\Var_k$ which includes all of the $\M$-morphisms, but only isomorphisms in the $\E$-direction.}}
\end{example}

\begin{remark}
   Any CGW-category $\C$ as defined in \cite{CZ-cgw} is also a squares category, as the axioms of a CGW-category assume strictly more information about compatibility than a squares category does. \autoref{ex:finset} and \autoref{ex:varcgw} are examples of CGW categories  constructed in  \cite{CZ-cgw}.
\end{remark}

We now introduce a new example of a squares category of varieties, whose $K$-theory will be studied in \autoref{sect:bittner}. First, we recall the concept of piecewise isomorphism of varieties which is used in the definition of distinguished squares. 

\begin{definition}\label{def:varpiecewise}
    A morphism of varieties $f:X \to Y$ is called a \textit{piecewise isomorphism}\footnote{Some references, including \cite{liusebag}, use a more general isomorphism which does \emph{not} assume the existence of a morphism $f: X \rto Y$, but using stratifications on $X$ and $Y$ separately and requiring isomorphisms on strata. This is not equivalent to the definition we are using.} if there exists a stratifications of closed immersions
\begin{diagram}
    { 
    Y_0 & Y_1 & \dots & Y_n=Y\\};
    \cofib{1-1}{1-2} \cofib{1-2}{1-3} \cofib{1-3}{1-4}
       
\end{diagram}
    such that
    \begin{enumerate}
        \item[(1)] the induced map $f^{-1}(Y_0) \to Y_0$ is an isomorphism,
        \item[(2)] for $i=1,\dots,n$, the map $f^{-1}(Y_i\smallsetminus Y_{i-1}) \to Y_i \smallsetminus Y_{i-1}$ is an isomorphism.
    \end{enumerate} 
\end{definition}

\begin{example}\label{ex:varkw}
    Let $\Var_{k,w}$ be the squares category with objects varieties over $k$, $\M_{\Var_{k,w}}$ the \edit{open\sout{closed}} immersions and $\E_{\Var_{k,w}}$ all morphisms of varieties.  The distinguished squares
 \begin{diagram}
     {A & B \\ C & D \\};
     \arrowsquare{\circ}{}{}{\circ}
 \end{diagram}are the squares that are pullbacks in the category of varieties over $k$, such that the induced morphism $B \smallsetminus A \rto D\smallsetminus C$ is a piecewise-isomorphism. We need to check that squares compose both vertically and horizontally.  The composition of pullback squares, either horizontally or vertically, is also a pullback square, so we focus on the piecewise-isomorphism condition.  Piecewise-isomorphisms also compose, so squares automatically compose vertically.  To check horizontal composition it suffices to verify that piecewise-isomorphisms satisfy \emph{extension} relative to subtraction sequences, in the sense that given a diagram
    \begin{diagram}
        {A & B & C \\ D & E & F\\};
        \to{1-1}{1-2}^\circ \to{1-3}{1-2}
        \to{2-1}{2-2}^\circ \to{2-3}{2-2}
        \to{1-1}{2-1}_f \to{1-2}{2-2}^g \to{1-3}{2-3}^h
    \end{diagram}
    where $f$ and $h$ are piecewise-isomorphisms, so is $g$.  (This holds by concatenating the stratifications on $D$ and $F$ appropriately.) Thus $\Var_{k,w}$ is well-defined.
\end{example}

\begin{example} \label{ex:polytope}
  Define a \emph{simplex} in $\R^m$ to be a convex hull of $n \leq m$ points.
  A \emph{polytope} is a finite union of simplices.  Let $\E_\C$ and $\M_\C$ both be the category
  whose objects are polytopes in $\R^m$ and whose morphisms are generated by
  isometric inclusions of polytopes. Squares are 
  distinguished if they are both pushouts and pullbacks after forgetting the double category structure. The basepoint is the empty polytope.
\end{example}

\begin{example}\label{manifoldex}
Let $\Mnfldbd_n$ be the category of smooth compact $n$-dimensional manifolds with boundary and smooth maps. We define both the subcategories $\E_\C$ and $\M_\C$ to be all manifolds but with morphisms given by the smooth embeddings of manifolds with boundary $f\colon N \rightarrow M$ such that $\partial N$ is mapped to a submanifold with trivial normal bundle, and such that each connected component of the boundary $\partial N$ is either mapped entirely onto a boundary component or entirely into the interior of $M$.
We define distinguished squares to be those commutative squares in $\Mnfldbd_n$     	
$$\begin{tikzcd}
    	N	\arrow[r ]\arrow[d ] &M\arrow[d]\\
    	M'\arrow[r] & M \cup_{N} M'.
    	\end{tikzcd}$$ that are pushout squares, i.e. such that $ M\cup_{N} M'$ is a smooth manifold. The chosen basepoint object is the empty manifold. This squares category is considered in \cite{WITMona}.

\end{example}

\begin{example}
    Let $\C$ be a squares category.  Then $\C^t$, the \emph{transpose} of $\C$ has $\M_{\C^t} = \E_\C$, and $\E_{\C^t} = \M_\C$.  A square is distinguished if it is distinguished in $\C$.
\end{example}

In \autoref{sect:bittner} we will see more examples of squares categories of varieties whose $K$-theories we will consider. 

\begin{definition}\label{defn:symm_mon}
    A squares category $\C$ is  \textit{symmetric monoidal} if both $\E_\C$ and $\M_\C$ have the structure of a symmetric monoidal category with unit $O$, which coincides on objects; in other words, such that $A \otimes_{\E_\C}B = A \otimes_{\M_\C} B$, and the product of two distinguished squares is a distinguished square. A symmetric monoidal functor of squares categories $F: \mathcal{C} \to \mathcal{D}$ is a functor of squares categories such that $F$ is a symmetric monoidal functor when restricted to both $\E_{\C} \to \E_{\D}$ and $\M_{\C} \to \M_{\D}$. 
\end{definition}

\begin{rmk}\label{rmk:restrictecoprod}
There is an important subtlety here involving the symmetric monoidal structure and coproducts.  In all of the examples that follow, it appears as if the symmetric monoidal structure is simply a coproduct.  However, this is not the case, as in most of these examples coproducts do not exist.  Consider, for instance, the category of finite sets and injections. This does not have coproducts, because the map $A \amalg A \rto A$ is not injective.  However, it does have an object with underlying set $A \amalg A$ and canonical injections $i_0,i_1: A \rcofib A\amalg A$.  The existence of this object can be formalized using the notion of \emph{restricted pushout}.
    Given a diagram $\mathcal{J}$,
    \[B \lto A \rto C,\]
    let $\mathbf{Cocone}_{\mathcal{J}}$ be the category of cocones $B \rto D \lto C$ under this diagram satisfying the property that the square
    \begin{diagram}
        { A & C \\ B & D \\};
        \arrowsquare{}{}{}{}
    \end{diagram}
    commutes and is a pullback square.  The \emph{restricted pushout} of the above diagram is then the initial object of $\mathbf{Cocone}_{\mathcal{J}}$.   
    
The category of finite sets and injections does have restricted pushouts. Considering the discussion above, the square
\begin{diagram}
    { \emptyset & A \\ A & A \\};
    \arrowsquare{}{}{}{}
\end{diagram}
is not a pullback square, and therefore does not appear in $\mathbf{Cocone}$.  Thus the non-injectivity of the fold map is not an obstruction to the restricted pushout existing. 
\end{rmk}

In all of the examples above, the symmetric monoidal structure is given either by the coproducts in $\E_\C$ and $\M_\C$, or by restricted pushouts, as explained in \autoref{rmk:restrictecoprod}. We prefer to separate out squares categories and symmetric monoidal squares categories, even though in this paper we only consider symmetric monoidal squares categories. However, there are examples of squares categories \textit{without} symmetric monoidal structure that could provide interesting insights in future work. 

\begin{remark}
    \edit{We omit the example of general SW categories of \cite{campbell}, as constructing a squares category from an SW category in a similar way to \autoref{ex:wald} does not work in a straightforward manner without an offensively long list of extra assumptions.  The problem with this construction is that SW categories lack a universal property for the complement, and also that there are not enough assumptions about how weak equivalences interact with pullbacks and with complements.}
\end{remark}

\section{\texorpdfstring{$K$}{K}-theory of squares categories}

We are now ready to define the $K$-theory of a squares category.

\begin{definition}
    Let $\C$ be a squares category and $\D$ an ordinary category.  A \textsl{horizontal functor} $\D \to \C$ is a functor $F:\D \to \M_\C$.  A \textsl{vertical distinguished transformation} between horizontal functors $F,G: \D \to \C$, denoted $\alpha: F \Rto G$, is a choice of a vertical morphism $\alpha_A: F(A) \rfib G(A)$ for all $A\in \D$ such that for any horizontal morphism $f: A \rcofib B$ in $\D$, the square
    \begin{diagram}
        { F(A) & F(B) \\ G(A) & G(B) \\};
        \cofib{1-1}{1-2}^{F(f)} \cofib{2-1}{2-2}^{G(f)}
        \fib{1-1}{2-1}_{\alpha_A} \fib{1-2}{2-2}^{\alpha_B}
    \end{diagram}
    is distinguished.

    The \textsl{category of horizontal functors}, denoted $\hFun(\D,\C)$ has objects the horizontal functors and morphisms the vertical distinguished transformations.  Analogous definitions of vertical functor, horizontal distinguished transformation, and category of vertical functors can be defined.
\end{definition}

\begin{definition} \label{def:Ksq}
  As usual, write $[k]$ for the category $0 \rto 1 \rto \cdots \rto k$. Let $\C$ be a squares category.  Let \[T_n\C = \hFun([k],\C).\]  
  These categories assemble into a
  simplicial category which we denote $T_\dotp \C$. On objects, for $0<i<1$, the $i$th face map is given by composing in the $i$th slot, and on morphisms it is given by composing the squares. The first and last face maps are given by deleting. The degeneracies are given by inserting identity maps on objects and repeating vertical maps on morphisms. 
 We define the squares $K$-theory of $\C$ by looping the geometric realization of the bisimplicial set given by the nerve of this simplicial category.
  \[K^\square(\C) = \Omega_O |N_\dotp T_\dotp\C|.\]
  Here, $\Omega_O$ is the based loop space based at the object $O \in N_0
  T_0\C$.  
\end{definition}

\begin{remark}
    Note that the nerve of the simplicial category $T_\dotp\C$ is a bisimplicial set, which by definition is also the nerve  of the double category $\C$, which we denote $N^\square_{\dotp,\dotp} \C$. We can also explicitly interpret the $(k,\ell)$-simplices as follows. Consider the category $[k]\times [\ell]$, which has a unique morphism from $(i,j)$ to $(i',j')$ exactly when $i\leq i'$ and $j\leq j'$. View this as a squares category $[k]\boxtimes [\ell]$ with horizontal morphisms given by those for which $i=i'$, vertical morphisms given by those for which $j=j'$, and with distinguished squares given by the commuting squares with these vertical and horizontal maps. Then the $(k,\ell)$-simplices $N^\square_{k,\ell} \C$ are the functors of squares categories $[k]\boxtimes [\ell] \rto \C$, and the bisimplicial structure is given by precomposition in the corresponding coordinate.
\end{remark}



The lemma below follows essentially by definition. 

\begin{lemma}\label{lem:space_functoriality}
A functor of squares categories $\mathcal{C} \to \mathcal{D}$ induces a map on $K$-theory spaces $$K^\square (\mathcal{C}) \to K^\square (\mathcal{D}).$$ 
\end{lemma}

Next we show that these constructions lift to spectra. Let $\sO$ be the categorical $E_\infty$ operad, also known as the Barratt-Eccles or the permutativity operad.\footnote{Sometimes this operad is also referred to as the ``chaotic categorification" of the associativity operad, which is the set valued operad with level $n$ given by the symmetric group $\Sigma_n$.} Each $O(n)$ is defined as the category with objects the elements of the symmetric group $\Sigma_n$ and with a unique morphism between any two objects. Algebras over $\sO$ in $\Cat$ are precisely the permutative categories (strictly unital and stricty associative symmetric monoidal categories) \cite{Maypermutative}. Every symmetric monoidal category is naturally monoidally equivalent to a permutative category \cite{isbell, GMMO}, and we will implicitly strictify symmetric monoidal categories to permutative ones. We consider strict maps of symmetric monoidal categories.\footnote{We could more generally consider strong monoidal functors which would induce multifunctors on the associated multicategories. Then the left adjoint of the forgetful functor from permutative categories to multicategories turns multifunctors into strict functors of permutative categories \cite{Tonymulti}.}

Note that $|\sO|$ is an $E_\infty$ operad in spaces, since $|\sO|\simeq E\Sigma_n$. Therefore, since the classifying space preserves products, the classifying space of any $\sO$-algebra $\C$ in $\Cat$ is an $E_\infty$ space, namely its group completion is an infinite loop space whose deloopings are given by the operadic infinite loop space machine \cite{Mayloop} (or equivalently by the Segal $\Gamma$-space machine \cite{MayThomason}). Now note that for a simplicial $\sO$-algebra $\C_{\sbt}$, the classifying space $|N_{\sbt}\C_{\sbt}|$ is then  an $|\sO|$-algebra in spaces, so it is an $E_\infty$ space. Therefore we get the following proposition.

\begin{prop}\label{Einfty}
If $\C_{\sbt}$ is a simplicial symmetric monoidal category, the geometric realization  $|N_{\sbt}\C_{\sbt}|$ is an $E_\infty$ space. Moreover, a map of simplicial symmetric monoidal categories induces a map of $E_\infty$ spaces.
\end{prop}

 This allows us to lift the squares $K$-theory construction to the spectrum level.  


\begin{thm}\label{infiniteloop}
    The $K$-theory space $K^\square(\C)$ for a symmetric monoidal squares category $\C$ is an infinite loop space. Moreover, a symmetric monoidal functor of squares categories  $\mathcal{C} \to \mathcal{D}$ induces a map of infinite loop spaces $K^\square(\mathcal{C}) \to K^\square (\mathcal{D})$. 
\end{thm}

\begin{proof}
    A symmetric monoidal structure on the squares category $\C$ induces a symmetric monoidal structure on each $T_n\C$ as follows. On the objects $\ob T_n\C= \Fun([k], \M_{\C})$ it is given by the pointwise symmetric monoidal structure in $\M_{\C}$. On morphisms, which are vertical distinguished transformations it is induced by the symmetric monoidal structure in $\E_{\C}$ and the fact that the product of two distinguished squares is again distinguished. Note that all face and degeneracy maps are symmetric monoidal. Therefore, $T_{\sbt}\C$ is a simplicial $\sO$-algebra  in $\Cat$. Thus by \autoref{Einfty} the space $|N_{\sbt}T_{\sbt} \C|$ is an $E_\infty$ space. \edit{Moreover, since $\E_\C$ and $\M_\C$ have an initial object (in fact it suffices that either of them has one), $|N_{\sbt}T_{\sbt}\C|$ is connected. Therefore it is grouplike, so it is an infinite loop space.}

From \autoref{defn:symm_mon}, a symmetric monoidal functor of squares categories $F: \mathcal{C} \to \mathcal{D}$ is a functor of squares categories such that $F$ is a symmetric monoidal functor when restricted to both $\E_{\C} \to \E_{\D}$ and $\M_{\C} \to \M_{\D}$. In particular, this induces symmetric  monoidal functors $T_n\C\to T_n\D$ for each $n$, which commute with the face and degeneracy maps. Thus again by \autoref{Einfty} we get a map of $E_\infty$ spaces.
\end{proof}

As we saw in the previous proof, we do not need the full strength of a symmetric monoidal functor on the level of squares categories in order to get the desired map of $E_\infty$ spaces, but it is enough to get maps of symmetric monoidal categories at the level of the $T_{\sbt}$ construction. We record the following criterion for getting an equivalence of infinite loop squares $K$-theory spaces which will be useful later. 

\begin{cor}\label{Einftycriterion}
    Let $\C$ and $\D$ be squares categories.  A simplicial functor of symmetric monoidal categories $T_{\sbt} \C \rto T_{\sbt}\D$ which is an equivalence $T_n \C \to T_n\D $ for all $n$, induces an equivalence of infinite loop spaces $K^\square(\C) \rto K^\square(\D)$. 
\end{cor}
In particular, a symmetric monoidal functor of squares categories for which the maps $T_n \C \rto T_n \D$ are equivalences gives  an equivalence of squares $K$-theory spectra. 

Since the operadic infinite loop space machine functorially produces the deloopings of the group completion of an $E_\infty$-algebra, from now on we sometimes abuse notation  and use $K^\square(\C)$ to refer to the associated $\Omega$-spectrum instead of just the underlying infinite loop space.

\begin{example}
    Consider \autoref{ex:finset}.  In this case, $K^\square(\mathbf{FinSet})$ is exactly the same as the $K$-theory of $\mathbf{FinSet}$ considered as a CGW-category in \cite[Example 3.2]{CZ-cgw}.  As shown there, this is equivalent to the usual $K$-theory of finite sets, which gives the sphere spectrum.
\end{example}

\begin{example}
    Consider \autoref{ex:polytope}.  In this case, $K_0^\square$ is McMullen's polytope algebra.  Conjecturally, it should also be the same as the scissors congruence $K$-theory of polyhedra $K(\mathfrak{G})$ defined using assemblers in \cite[Section 5.2]{Z-Kth-ass}, although a direct comparison is currently unknown.  The comparison in \autoref{sect:bittner} for varieties uses strongly that the inclusion of a variety into its compactification is an open immersion, whose analogue does not hold for polytopes. 
\end{example}

\begin{example}
    Consider \autoref{manifoldex}. In this case, $K^\square(\Mnfldbd_n)$ is the $K$-theory of manifolds. \autoref{infiniteloop} shows that the squares $K$-theory space constructed in \cite{WITMona} is indeed a spectrum. Its $K^\square_0$ is the scissors congruence group $SK^\partial_n$ for $n$-dimensional manifolds with boundary.
\end{example}

\begin{example}
    Let $\C$ be a squares category, and let $\C^t$ be the transpose, where the roles of $\M$-morphisms and $\E$-morphisms are switched.  Then $K^\square(\C) \cong K^\square(\C^t)$, as the diagonals of the bisimplicial sets $N^\square_{\dotp,\dotp} \C$ and $N^\square_{\dotp,\dotp}\C^t$ are naturally isomorphic.
\end{example}

\begin{proposition}\label{lem:wald}
  Let $\C$ be a Waldhausen category, and let $\C^\square$ be the associated squares category defined in \autoref{ex:wald}. Then
  \[K^W(\C) \simeq K^\square(\C^\square).\]
  Here $K^W$ is the usual Waldhausen $K$-theory of $\C$. 
\end{proposition}

\begin{proof}
  The simplicial category $T_{\sbt} \C^\square$ is exactly the Thomason construction $wT_{\sbt} \C$ in \cite[end of section 1.3]{waldhausen}, and is shown to be homotopy equivalent to the $S_{\sbt}$-construction $wS_{\sbt} \C$. 
  %
\end{proof}

\begin{proposition} \label{prop:oneKvar}
Let $k$ be a field. The natural inclusions of squares categories $\Var_k \rto \Var_{k,w}$ is a weak equivalence after applying $K^\square$.  Moreover,\edit{\sout{$K^\square(\Var_k^i)$,}} $K^\square(\Var_k)$ and $K^\square(\Var_{k,w})$ agree with previously-defined versions of the Grothendieck spectrum of varieties, including the model of assemblers \cite[Section 5.1]{Z-Kth-ass}, \edit{the model of subtractive categories} \cite{campbell}, and \edit{the model of CGW categories} \cite[Example 3.4]{CZ-cgw}.
\end{proposition}

\begin{proof}
    We begin by proving that $K^\square(\Var_{k,w})$ agrees with previously-defined models.
    By \cite[Theorem 9.1]{CZ-cgw} all of the previously-defined versions of the Grothendieck spectrum of varieties agree, so it suffices to focus on one model.  The \edit{category of varieties, which is the underlying category of the squares} category $\Var_{k,w}$,\edit{\sout{also}}has an SW-category structure \edit{with cofibrations given by the class $\M_{\Var_{k,w}}$, i.e., the open immersions, and weak equivalences given by piecewise isomorphisms; see} \cite[Lemma 9.11]{CZ-cgw}. \edit{We denote this SW-category also by $\Var_{k,w}$}.  We can compare \edit{the $T_\dotp$-construction for the squares category $\Var_{k,w}$} with the $\tilde S_\dotp$-construction for the SW category $\Var_{k,w}$ from \cite[Definition 3.32]{campbell}.  The proof here is an adaptation of the proof from \cite[end of Section 1.3]{waldhausen} that the Thomason construction and the Waldhausen constructions agree \edit{for Waldhausen categories, but the argument becomes even more straightforward since there are canonical choices of complements in varieties.}  

    To  show that $|N_\dotp T_\dotp\Var_{k,w}|$ and $|w\tilde S_\dotp \Var_{k,w}|$ are weakly equivalent, we define a \sout{\edit{intermediary}}simplicial category $T^+_{\sbt} \Var_{k,w}$ \edit{which is isomorphic} to $T_{\sbt} \Var_{k,w}$ \edit{and has a map to  $w\tilde S_\dotp \Var_{k,w}$,} as follows.  
    Let $\operatorname{Ar}^\square[n]$ be the squares category with the following structure.
    \begin{description}
        \item[objects] Non-identity morphisms in $[n]$, i.e., relations $(i<j)$, for $0\leq i < j \leq n$, and a basepoint object $e$. (We can think of $e$ as all the identity morphisms identified.) 
        \item[$\M$-morphisms] For every triple $a<b<c$ in $[n]$ there is an $\M$-morphism $(a<b) \to (a<c)$, and the object $e$ is initial in $\M$.
        \item[$\E$-morphisms] For every triple $a<b<c$ in $[n]$ there is an $\E$-morphism $(b<c) \to (a<c)$ and the object $e$ is initial in $\E$.
        \item[distinguished squares] For every quadruple $a < b\leq c < d$ there is a distinguished square
        \begin{diagram}
            { (b<c) & (b<d) \\ (a<c) & (a<d). \\};
            \to{1-1}{1-2} \to{2-1}{2-2} \to{1-1}{2-1} \to{1-2}{2-2}
        \end{diagram}
        when $b<c$, and a similar square with the object $e$ in the upper left corner when $b = c$.
    \end{description}

An object of $T^+_n \Var_{k,w}$ is a squares functor $X:\operatorname{Ar}^\square[n+1] \rto \Var_{k,w}$. We can visualize an object in $T^+_n \Var_{k,w}$ as a diagram
  \begin{diagram}
     {& & & & \emptyset \\
     & & & \emptyset & X_{n<n+1} \\
     & & \dots &\dots & \dots\\
    &  \emptyset & X_{1<2} &\dots & X_{1<n+1} \\
    \emptyset & X_{0<1} & X_{0<2} & \dots & X_{0<n+1}\\ };
  \to{2-4}{2-5}   \to{4-2}{4-3} \to{4-3}{4-4} \to{4-4}{4-5}
    \to{5-1}{5-2} \to{5-2}{5-3} \to{5-3}{5-4} \to{5-4}{5-5} \to{1-5}{2-5}  \to{2-4}{3-4} \to{2-5}{3-5}  \to{3-3}{4-3} \to{3-5}{4-5} \to{4-2}{5-2} \to{4-3}{5-3} \to{4-5}{5-5}
  \end{diagram}
    where \edit{the horizontal morphisms are open immersions, the vertical morphisms are closed immersions, and} all the squares are distinguished squares in $\Var_{k,w}$. We can think of such an object as an object of $T_{n}\Var_{k,w}$ (the bottom row of the diagram, without the first object $\emptyset$) together with \edit{\sout{a} the canonical} choice of complement $X_{i<j} = X_{0<j}\smallsetminus X_{0<i}$ for every \edit{\sout{closed}open} immersion $X_{0<i}\to X_{0<j}$.

A morphism \edit{$\alpha:X \rto X'$ of $T^+_n \Var_{k,w}$ is a collection of morphisms of varieties $\alpha_{ij}:X_{i<j} \to X'_{i<j}$ such that for all $i < j < k$ the left-hand square is a distinguished square in the squares category $\Var_{k,w}$ and the right-hand square commutes in the ordinary category of varieties.}
\begin{diagram}
    {X_{i<j} & X_{i<k} & \qquad & X_{j<k} & X'_{j<k} \\
    X'_{i<j} & X'_{i<k} & & X_{i<k} & X'_{i<k} \\};
    \to{1-1}{1-2} \to{2-1}{2-2}
    \to{1-4}{1-5}^{\alpha_{jk}} \to{2-4}{2-5}^{\alpha_{ik}}
    \to{1-1}{2-1}_{\alpha_{ij}} \to{1-2}{2-2}^{\alpha_{ik}}
    \to{1-4}{2-4} \to{1-5}{2-5}
\end{diagram}
    Since distinguished squares induce piecewise isomorphisms on complements, the maps $\alpha_{ij}:X_{i<j}\to X'_{i<j}$ are piecewise isomorphisms for all $0<i<j\leq n$.  \edit{A morphism of $T_n^+\Var_{k,w}$ can thus be thought of as a morphism in $T_n\Var_{k,w}$, together with the induced morphisms on complements.}

Note that $T^+_{\sbt} \Var_{k,w}$ is a simplicial category, with the simplicial structure induced by precomposition in $\Delta$, where we drop the first face and degeneracy maps (since our level $n$ is shifted up from $\operatorname{Ar}^\square[n]$).

    There is a functor $[n] \rto \operatorname{Ar}^\square[n+1]$ given by taking $i$ to $0 \rightarrow i+1$.  Precomposition with this functor induces a map of simplicial categories $T_n^+\Var_{k,w} \rto T_n \Var_{k,w}$, taking the triangular diagram above to the bottom row \edit{\sout{without the first object}with the first object $\emptyset$ deleted.} This is levelwise an \edit{\sout{equivalence}isomorphism}of categories, and therefore induces an equivalence after applying $|N_\dotp\cdot|$. Thus it remains to show that there exists a simplicial functor $t_\dotp\colon T_\dotp^+\Var_{k,w} \rto w\tilde S_\dotp \Var_{k,w}$ which induces an equivalence after applying $|N_\dotp\cdot|$.  
    
    The simplicial functor $t_\dotp$ takes a functor $X$ in $T_\dotp^+ \Var_{k,w}$ to the functor $t_\dotp(X): \operatorname{Ar}[n] \rto \Var_{k,w}^{\op}$ defined by $t_\dotp(X)(i<j) = X_{i+1<j+1}$. 
    Visually, what happens to an object represented by a triangular diagram as above is that the bottom row is removed, and the rest is shifted down and to the left. By definition and the properties of $F$ this gives a well-defined object of $w\tilde S_\dotp \Var_{k,w}$.
    
    Since the components of morphisms in $T^+_n\Var_{k,w}$ are piecewise isomorphisms in all levels above the bottom row, we have that the image of $t_\dotp$ lands inside $w\tilde S_\dotp \Var_{k,w}$, and is therefore a functor of simplicial categories.  It remains to check that this induces a weak equivalence on geometric realization.

    We claim that $t_n$ has a homotopy inverse $G_n: w \tilde S_n \Var_{k,w} \rto T^+_n \Var_{k,w}$ defined by 
    \[G_n(F)(i<j) = \left\{\begin{array}{ll} F(i-1<j-1) & i > 0 \\
    F(0,j-1) & i = 0 \textup{ and } j\geq 2 \\
    \emptyset & i=0 \textup{ and } j =1
    
    \end{array}\right.\]
    Visually, it takes an object in $w\tilde{S}_n\Var_{k,w}$, shifts it up and to the right, repeats the bottom row and adds an $\emptyset$ on the left.
    
    By definition, $t_nG_n = 1_{w\tilde S \Var_{k,w}}$.

    A natural transformation $\epsilon:G_n\circ t_n \Rto \id_{T^+_n\Var_{k,w}}$ is given by 
    \[\epsilon_{i< j} =  \left\{\begin{array}{ll}
    \id_{X_{i<j}} & i > 1 \\
    X_{1<j} \rto X_{0 < j}  & i = 0
     \end{array}\right.\]
     Thus $t_\dotp$ has a levelwise homotopy inverse, as desired.

    There \edit{\sout{are}is a} natural inclusion of squares categories\edit{\sout{$\Var_k^i \rto$}} $\Var_k \rto \Var_{k,w}$, since all squares in $\Var_k$ are pullbacks by definition and induce isomorphisms on complements.  We can define \edit{a} wide subcategor\edit{y\sout{ies $T^+_\dotp \Var^i_k$ and}} $T^+_\dotp \Var_k$ of $T^+_\dotp \Var_{k,w}$, with the same objects, but morphisms for which the naturality squares between bottom rows are distinguished in \edit{\sout{$\Var_k^i$ and}}$\Var_k$. \edit{\sout{respectively}}\edit{The functor $t_\dotp$ restricts to a functor $T_\dotp^+ \Var_k \rto i\tilde S_\dotp \Var_k$, which is also a homotopy equivalence since each $G_n$ restricts to a functor $w\tilde S_n\Var_{k} \to T^+_n\Var_k$.}  Consider the following diagram of simplicial categories:
    \begin{diagram}
        { 
        T_\dotp\Var_k & T_\dotp^+ \Var_k & i \tilde S_\dotp \Var_k & i\tilde S_\dotp\mathbf{Sch}_{rf} \\
        T_\dotp\Var_{k,w} & T_\dotp^+ \Var_{k,w} & w \tilde S_\dotp \Var_{k,w} & w\tilde S_\dotp\mathbf{Sch}_{rf,w} \\
        };
        \to{1-2}{1-1}_{\edit{\cong}} \to{1-2}{1-3}^{t_\dotp} \to{1-3}{1-4}
        \to{2-2}{2-1}_{\edit{\cong}} \to{2-2}{2-3}^{t_\dotp} \to{2-3}{2-4}
        \to{1-1}{2-1} \to{1-2}{2-2} \to{1-3}{2-3} \to{1-4}{2-4}
    \end{diagram}
    Here, $\mathbf{Sch}_{rf}$ is the SW-category of reduced schemes of finite type, considered as an SW-category, with either isomorphisms (on the top row) or piecewise-isomorphisms (on the bottom row) as the weak equivalences.   On geometric realizations, the rightmost vertical morphism in this diagram is a weak equivalence by \cite[Theorem 9.1]{CZ-cgw}.  The leftmost and central horizontal morphisms are weak equivalences by the argument above.  The rightmost horizontal morphisms are weak equivalences on geometric realization by \cite[Corollary 7.11]{CZ-cgw}.  Thus the leftmost vertical morphism is also a weak equivalence on geometric realization, as desired.
\end{proof}

Motivated by this comparison, we introduce the following notation.

\begin{definition}
    The spectrum $K(\Var_k)$ is the \emph{Grothendieck spectrum of varieties}.  For the purposes of this paper we will model it as \edit{either $K^\square(\Var_k)$ or $K^\square(\Var_{k,w});$} due to the above proposition, it agrees with the constructions in \cite{Z-Kth-ass,campbell,CZ-cgw}.
\end{definition}

Before we end this section we introduce one more squares category of varieties, which will be useful in \autoref{sect:bittner} when we discuss Bittner's presentation \edit{of the Grothendieck ring of varieties in the case when $k$ has characteristic 0}.

\begin{definition}\label{Varkpw}
    Let $\Var_{k,pw}$ be the squares subcategory of $\Var_{k,w}$ containing all objects, all $\M$-morphisms, and only the proper $\E$-morphisms. 
\end{definition}


\edit{We obtain the following corollary of \autoref{prop:oneKvar}.}

\begin{cor}\label{Varkpw_retract}
    $K^\square(\Var_{k,pw})$ is (up to weak equivalence) a retract of $K(\Var_k)$.
\end{cor}

\edit{
\begin{proof}
Since all closed immersions are proper, we have inclusions of squares categories 
$$\Var_k \rcofib \Var_{k,pw} \rcofib \Var_{k,w}$$
and by \autoref{prop:oneKvar}, the composition induces an equivalence of $K$-theory spectra
$K^\square(\Var_{k}) \simeq K^\square(\Var_{k,w})$. Thus $K^\square(\Var_{k,pw})$ is a retract of $K(\Var_k)$. 
\end{proof}
}



\section{Computation of \texorpdfstring{$K_0$}{K0}}

In this section we give a computation of $K_0^\square(\C)$ for a squares category  $\C$. In order to show that we do get the free abelian group on the objects, modulo exactly the relation that breaks up the square relation, we need to impose some extra conditions on the input squares category.

\begin{thm} \label{lem:K0norm}
  Let $\C$ be a squares category satisfying the following extra condition: 
  
  ($\ast$) For all objects $A,B\in \C$ there exists some object $X$ and
    distinguished squares
    \[
      \begin{inline-diagram}
        { O & A \\ B & X \\}; \cofib{1-1}{1-2} \cofib{2-1}{2-2} \cofib{1-1}{2-1}
        \cofib{1-2}{2-2} 
      \end{inline-diagram}\qqand
      \begin{inline-diagram}
        { O & B \\ A & X \\}; \cofib{1-1}{1-2} \cofib{2-1}{2-2} \fib{1-1}{2-1}
        \fib{1-2}{2-2}       
      \end{inline-diagram}.\]

  Then
  \[K_0(\C) \cong \Z\{\ob \C\}/\sim,\] where $\sim$ is the relation that
  $[O] = 0$ and for every distinguished square
  \[\begin{inline-diagram}
    { A & B \\ C & D \\};
    \cofib{1-1}{1-2} \cofib{2-1}{2-2} \fib{1-1}{2-1} \fib{1-2}{2-2}
  \end{inline-diagram}
  \qquad \hbox{we have} \qquad [A] + [D] = [B] + [C].\]
\end{thm}

\begin{proof}
  Let $X = |N_{\sbt} T_{\sbt} \C|$, so that $K(\C) = \Omega_OX$. The $1$-simplices of
  $X$ are exactly the distinguished squares of $\C$.  Let $T$ be the subspace of
  $X$ given by all $0$-simplices, and by the $1$-simplices of the form
  \begin{diagram}
    { A & B \\ A & B \\};
    \eq{1-1}{2-1} \eq{1-2}{2-2}
    \cofib{1-1}{1-2} \cofib{2-1}{2-2}
  \end{diagram}
  and all higher-dimensional simplices having these as its boundary.
  Then $T$ is contractible
  (since $T$ is isomorphic to the nerve of $\M_C$, and $O$ is initial in $\M_\C$) and thus $X/T \simeq X$.
  On the other hand, $X/T$ has only one $0$-simplex, so $\pi_1(X/T)$ is
  generated by all $1$-simplices of $X/T$.  Consider the $2$-simplex of $X$ of
  the form
  \begin{equation*}
    \begin{inline-diagram}
      {A & A & B \\
        C & C & D \\
        C & C & D \\};
      \eq{1-1}{1-2} \eq{2-1}{2-2} \eq{3-1}{3-2}
      \eq{2-1}{3-1} \eq{2-2}{3-2} \eq{2-3}{3-3}
      \fib{1-1}{2-1} \fib{1-2}{2-2} \fib{1-3}{2-3}
      \cofib{1-2}{1-3} \cofib{2-2}{2-3} \cofib{3-2}{3-3}
    \end{inline-diagram};
  \end{equation*}
  it induces the relation that
  \[\left[
      \begin{inline-diagram}
        { A & B \\ C & D \\}; \cofib{1-1}{1-2} \cofib{2-1}{2-2} \fib{1-1}{2-1} \fib{1-2}{2-2}
      \end{inline-diagram}\right] = \left[
      \begin{inline-diagram}
        { A & A \\ C & C \\}; \eq{1-1}{1-2} \eq{2-1}{2-2} \fib{1-1}{2-1} \fib{1-2}{2-2}
      \end{inline-diagram}\right];\]
  note that the square \edit{\sout{containing only cofibrations and equalities is contained}with equalities as vertical maps is}
  in $T$, and thus does not appear in the induced relation.  Considering an
  analogous argument applied to the square
  \begin{equation*}
    \begin{inline-diagram}
      {A & B & B \\
        A & B & B \\
        C & D & D \\};
      \eq{1-1}{2-1} \eq{1-2}{2-2} \eq{1-3}{2-3}
      \eq{1-2}{1-3} \eq{2-2}{2-3} \eq{3-2}{3-3}
      \cofib{1-1}{1-2} \cofib{2-1}{2-2} \cofib{3-1}{3-2}
      \fib{2-1}{3-1} \fib{2-2}{3-2} \fib{2-3}{3-3}
    \end{inline-diagram}
  \end{equation*}
  produces the relation
  \[\left[
      \begin{inline-diagram}
        { A & B \\ C & D \\}; \cofib{1-1}{1-2} \cofib{2-1}{2-2} \fib{1-1}{2-1} \fib{1-2}{2-2}
      \end{inline-diagram}\right] = \left[
      \begin{inline-diagram}
        { B & B \\ D & D \\}; \eq{1-1}{1-2} \eq{2-1}{2-2} \fib{1-1}{2-1} \fib{1-2}{2-2}
      \end{inline-diagram}\right].\]
  Therefore $\pi_1(X/T)$ is generated by squares \edit{\sout{containing only cofiber maps}whose horizontal maps are identities}; to
  save space, we denote such a square by $[A \rfib C]$.  The two relations above
  combine to give
  \[\makeshort{[A \rfib C] = [B \rfib D]}.\]
  Note that $X/T$ contains inside it the nerve of $\E_\C$ with all $0$-simplices
  identified.  Since $O$ is initial in $\E_{\C}$ we have the induced relation that
  for any \edit{\sout{cofiber}vertical} map $[A \rfib B]$,
  \[[O \rfib B][O \rfib A]^{-1} = [A \rfib B].\] Thus $\pi_1(X/T)$ is actually
  generated by generators of the form $[0 \rfib A]$ (all of which are
  nondegenerate, except $O$), which we write $[A]$ for short, and each
  distinguished square (as above) gives the relation that
  \begin{equation} \label{eq:relation} [C][A]^{-1} = [D][B]^{-1}.
  \end{equation} In order to prove the lemma it therefore
  remains only to prove that $\pi_1(X/T)$ is abelian.  This follows from
  condition (3); for any $A$, $B$ the existence of $X$ and the two squares
  implies that $[B][A^{-1}] = [X] = [A][B]^{-1}$, as desired.  From this point on, we therefore switch to additive notation, and (\ref{eq:relation}) becomes 
  \begin{equation}
      \label{eq:relationagain} [C]-[A] = [D]-[B]
  \end{equation}

  Thus we have shown that $\pi_1(X)$ is a quotient of the desired group.  It remains to check that there are no other relations induced by $2$-simplices.  For any $2$-simplex $\sigma$ we have the additional relation that $d_0\sigma + d_2\sigma = d_1\sigma$.  Thus consider a general $2$-simplex 
  \[\sigma = \begin{inline-diagram}
      { A & B & C \\ D & E & F \\ G & H & I \\};
      \cofib{1-1}{1-2} \cofib{1-2}{1-3} 
      \cofib{2-1}{2-2} \cofib{2-2}{2-3}
      \cofib{3-1}{3-2} \cofib{3-2}{3-3}
      \fib{1-1}{2-1} \fib{2-1}{3-1}
      \fib{1-2}{2-2} \fib{2-2}{3-2} 
      \fib{1-3}{2-3} \fib{2-3}{3-3}
  \end{inline-diagram}\]
  in $X$.  We have 
  \[d_0\sigma = \begin{inline-diagram} {E & F \\ H & I \\}; \cofib{1-1}{1-2} \cofib{2-1}{2-2} \fib{1-1}{2-1} \fib{1-2}{2-2} \end{inline-diagram}
  \qquad 
  d_1\sigma = \begin{inline-diagram} {A & C \\ G & I \\}; \cofib{1-1}{1-2} \cofib{2-1}{2-2} \fib{1-1}{2-1} \fib{1-2}{2-2} \end{inline-diagram}
  \qquad 
  d_2\sigma = \begin{inline-diagram} {A & B \\ D & E \\}; \cofib{1-1}{1-2} \cofib{2-1}{2-2} \fib{1-1}{2-1} \fib{1-2}{2-2} \end{inline-diagram}.\]
  In $\pi_1(X/T)$, using the relation (\ref{eq:relationagain}) to simplify , produces the equation 
  \[([E]+[I] - [F]-[H]) + ([A] + [E] - [B] - [D]) = [A] + [I] - [C] - [G].\]
  This simplifies to 
  \[[F] +[H] +[B] +[D] = 2[E] +[C] + [G].\]
  Using (\ref{eq:relationagain}) the two off-diagonal squares give the two relations 
  \[[D]+[H] = [E] + [G] \qqand [B] + [F] = [C] + [E];\]
  as the above is simply the sum of the two of these, we see that the relation induced by $\sigma$ does not induce any further relations in $\pi_1(X/T)$, and the result follows.
\end{proof}
\begin{remark} $\hbox{ }$
\begin{enumerate}
  \item If $\C$ is a squares category which does not satisfy condition ($\ast$) from \autoref{lem:K0norm}, then $K_0(\C)$ is still generated by the objects of $\C$ which satisfy the equivalence relation imposed by the squares. However, then we cannot ensure that $K_0(\C)$ is abelian anymore. There are interesting examples of squares categories with such nonabelian $K_0$-group.
\item  If the symmetric monoidal structure on $\C$ is given by finite coproducts, the empty coproduct by definition is an initial object $O$. If we require that the maps from the initial object to any other object are in both $\M_\C$ and $\E_\C$, then condition ($\ast$) in \autoref{lem:K0norm} is also automatically satisfied: by definition we are then requiring the following two squares to be distinguished 
 \[ \begin{inline-diagram} {O & A \\ O & A \\}; \cofib{1-1}{1-2} \cofib{2-1}{2-2} \eq{1-1}{2-1} \eq{1-2}{2-2} \end{inline-diagram}
  \qquad 
   \begin{inline-diagram} {O & O \\ B &  B \\}; \eq{1-1}{1-2} \eq{2-1}{2-2} \fib{1-1}{2-1} \fib{1-2}{2-2} \end{inline-diagram}\]
	and since squares are  required to preserve coproducts, the following squares are then also distinguished:
  \[ \begin{inline-diagram} {O & A \\ B & A\amalg B \\}; \cofib{1-1}{1-2} \cofib{2-1}{2-2} \fib{1-1}{2-1} \fib{1-2}{2-2} \end{inline-diagram}
  \qquad 
   \begin{inline-diagram} {O & A \\ B & B\amalg A \\}; \cofib{1-1}{1-2} \cofib{2-1}{2-2} \fib{1-1}{2-1} \fib{1-2}{2-2} \end{inline-diagram}.\]
   Reversing the roles of $A$ and $B$ proves condition $(\ast)$.

\end{enumerate}
\end{remark}

\section{Towards Bittner's Presentation of the Grothendieck spectrum of varieties}\label{sect:bittner}

\edit{Classically, the Grothendieck ring of varieties is defined as the free abelian group on isomorphism classes of varieties, subject to the relation $[X-Y]=[X]-[Y]$ for a closed subvariety $Y$ of $X$.} In \cite{bittner04}, Bittner proved an alternate presentation of the Grothendieck ring of varieties over a field with characteristic zero.  In particular, she proved the following:
\begin{thm}
    Let $k$ be a field of characteristic $0$.  Then the Grothendieck ring of varieties is generated by smooth projective varieties over $k$.  The only relation necessary is that for any square
    \begin{diagram}
        { E & \mathrm{Bl}_Y X \\ Y & X \\};
        \cofib{1-1}{1-2} \cofib{2-1}{2-2} \to{1-1}{2-1} \to{1-2}{2-2}
    \end{diagram}
    where $Y \rcofib X$ is a closed immersion, $\mathrm{Bl}_YX$ is the blowup of $Y$ along $X$, and $E$ is the exceptional divisor, there is a relation
    \[[X] - [Y] = [\mathrm{Bl}_Y X] - [E].\]
\end{thm}
Note that this is a four-term relation which is expressed via a square, so it is amenable to a squares formulation.  However, such squares do not compose, either vertically or horizontally, so some modification is necessary in order to make it into a squares category.  The goal of this section is to construct several possible spectral analogues to this presentation and explore their relationships.
    
\begin{definition}
    An \emph{abstract blowup square} is a cartesian square 
    \begin{diagram}
        { E & Y \\ C & X \\};
        \cofib{1-1}{1-2} \cofib{2-1}{2-2}
        \to{1-1}{2-1} \to{1-2}{2-2}^f
    \end{diagram}
    where $f$ is proper, the horizontal morphisms are closed immersions, and $f$ induces an isomorphism $Y \backslash E \rto X \backslash C$.  A \emph{formal blowup square} is a cartesian square
    \begin{diagram}
        { A & B \\ C & D \\};
        \cofib{1-1}{1-2} \cofib{2-1}{2-2} \to{1-1}{2-1} \to{1-2}{2-2}
    \end{diagram}
    which can be represented as the outside square of a diagram
    \begin{diagram}
        { X_0 & X_1 & \cdots & X_n \\
          X_0' & X_1' & \cdots & X_n' \\};
          \cofib{1-1}{1-2} \cofib{1-2}{1-3} \cofib{1-3}{1-4}
          \cofib{2-1}{2-2} \cofib{2-2}{2-3} \cofib{2-3}{2-4}
          \to{1-1}{2-1} \to{1-2}{2-2} \to{1-4}{2-4}
    \end{diagram}
    in which every square is an abstract blowup square. 
\end{definition}
\begin{lemma} \label{lem:formal-piecewise}
    A formal blowup square
    \begin{diagram}
        { A & B \\ C & D \\};
        \cofib{1-1}{1-2} \cofib{2-1}{2-2}
        \to{1-1}{2-1} \to{1-2}{2-2}
    \end{diagram}
    induces a proper piecewise isomorphism $B \backslash A \rto D \backslash C$.
\end{lemma}

\begin{proof}
    By definition, an abstract blowup square induces an isomorphism on the complements.  To prove the statement for formal blowup squares, we show that if we are given a diagram
    \begin{diagram}
        { A & B & C \\ D & E & F \\};
        \cofib{1-1}{1-2} \cofib{1-2}{1-3}
        \cofib{2-1}{2-2} \cofib{2-2}{2-3}
        \to{1-1}{2-1} \to{1-2}{2-2}^f \to{1-3}{2-3}^g
    \end{diagram}
    where each of the squares is a formal blowup square which induce proper piecewise isomorphisms $\tilde f:B \smallsetminus A \rto E \smallsetminus D$ and $\tilde g:C\smallsetminus B \rto F\smallsetminus E$, then the outside square induces a proper piecewise isomorphism $C \smallsetminus A \rto F \smallsetminus D$.  From this the desired result follows by induction.

    The fact that $\tilde f$ is a piecewise isomorphism means that there is a stratification $\emptyset = X_0 \subseteq \cdots \subseteq X_n = (B \smallsetminus A)$ such that $\tilde f|_{f^{-1}(X_i\smallsetminus X_{i-1})}$ is an isomorphism for all $i \geq 1$.  Similarly, there exists a stratification $\emptyset = Y_0 \subseteq \cdots \subseteq Y_m = (C \smallsetminus B)$ 
    such that $\tilde g|_{g^{-1}(Y_j \smallsetminus Y_{j-1})}$ is an isomorphism.

    To construct a stratification on $C \smallsetminus A$ we define
    \[Z_\ell \defeq \left\{\begin{array}{ll} 
        X_\ell & \hbox{if } 0 \leq \ell \leq n \\
        (C \smallsetminus A) \smallsetminus ((C \smallsetminus B) \smallsetminus Y_{\ell-n}) & \hbox{if } n+1 \leq \ell \leq m+n.
    \end{array}\right.\]
    Note that when $\ell=m+n$ this is exactly $C \smallsetminus A$.  Moreover, each inclusion is a closed immersion as it is the complement of the open immersion $(C\smallsetminus B) \smallsetminus Y_{\ell-n} \rto^\circ (C \smallsetminus B) \rto^\circ C \smallsetminus A$, so it is a valid stratification.
    For $\ell \leq n$, we have
    \[g|_{g^{-1}(Z_\ell \smallsetminus Z_{\ell-1})} = f|_{f^{-1}(X_\ell\smallsetminus X_{\ell-1})},\]
    because the right-hand square is a pullback square, and this is an isomorphism by the assumption on $f$.  On the other hand, for $n+1\leq \ell \leq m+n$ we have
    \[g|_{g^{-1}(Z_\ell \smallsetminus Z_{\ell-1})} = g|_{g^{-1}(Y_{\ell-n}\smallsetminus Y_{\ell-n-1}},\]
    which is an isomorphism by the assumption on $g$.  Thus the induced map is a piecewise-isomorphism, as desired.  It is proper because the underlying morphism is a pullback of a proper morphism.
\end{proof}

We note that abstract blowup squares are closed under vertical composition, but generally not under horizontal composition, so that all abstract blowup squares are formal blowup squares, but not vice versa.
  \begin{example}  Let $\mathbf{Comp}_k$ be the category of complete varieties over $k$. We consider $\mathbf{Comp}_k$ as squares category whose vertical morphisms are proper morphisms, horizontal morphisms are closed immersions, and the distinguished squares are the formal blowup squares (with the empty variety $\emptyset$ as the chosen $0$). 
\end{example}

 By \autoref{prop:oneKvar}, $K^\square(\Var_{k,w})\simeq K^\square(\Var_k)$, so in particular,  $K_0^\square(\Var_{k,w})$ is the underlying group of the Grothendieck ring of varieties. The group $K_0^\square(\Comp_k)$, which we can describe explicitly using \autoref{lem:K0norm}, is also isomorphic to the underlying group of the Grothendieck ring of varieties \cite[Proposition 7.10]{kuijper}. Therefore, a natural question to ask is if the spectra $K^\square(\Comp_k)$ and $K^\square(\Var_{k,w})$ are equivalent. 

\edit{Note that as squares categories, $\Comp_k$ is not a subcategory of $\Var_{k,w}$. Instead, we can hope to compare them at the level of the $T_{\sbt}$-construction.  However, if we tried to define a map on the level of $T_{\sbt}$-constructions $T_{\sbt} {\Comp_k}\to T_{\sbt} \Var_{k,w}$, sending an object 
$$X_0\rcofib X_1 \rcofib \dots \rcofib X_n$$ to
$$\emptyset \longrightarrow X_n\setminus X_{n-1} \xlongrightarrow{\circ} \dots \xlongrightarrow{\circ} X_n\setminus X_0,$$
this would fail to be simplicial. Instead we map to the simplicial category $w\tilde S_n \Var_{k,w}$ using the SW structure on $\Var_{k,w}$, since by \autoref{prop:oneKvar} the two constructions give equivalent spectra.
}
\begin{definition}
     Let
 \edit{\sout{$r_n: T_n\Comp_k \rto T_n\Var_{k,w}$}$r_n:T_n\Comp_{k}\to w\tilde S_n \Var_{k,w}$}
be the functor sending an object
    $$X_0 \rcofib X_1 \rcofib \dots \rcofib X_n $$
    to \edit{\sout{$\emptyset \longrightarrow X_n\setminus X_{n-1} \xlongrightarrow{\circ} \dots \xlongrightarrow{\circ} X_n\setminus X_0,$}}

    \begin{diagram}
        {\emptyset & X_n\smallsetminus X_{n-1} & X_{n}\smallsetminus X_{n-2} & \dots & X_n\smallsetminus X_0\\
        & \emptyset & X_{n-1}\smallsetminus X_{n-2} & \dots & X_{n-1}\smallsetminus X_0 \\ & &\dots & &  \dots\\
        & & & & X_{1}\smallsetminus X_0 \\
        & & & & \emptyset\\};
        \to{1-1}{1-2} \to{1-2}{1-3}^\circ \to{1-3}{1-4}^\circ \to{1-4}{1-5}^\circ
        \to{2-2}{1-2} \cofib{2-3}{1-3} \cofib{2-5}{1-5}
        \to{2-2}{2-3} \to{2-3}{2-4}^\circ \to{2-4}{2-5}^\circ \cofib{3-3}{2-3} \cofib{3-5}{2-5} \cofib{4-5}{3-5}         \cofib{5-5}{4-5}
    \end{diagram} 
\end{definition}

   
We observe that $r_n$ preserves coproducts. The functors $r_n$ assemble to a simplicial functor $r_{\sbt}:T{\edit{^\op}}_{\sbt}\Comp_k \to w\tilde{S}_{\sbt}\Var_{k,w}$, \edit{where $T^\op_{\sbt}\Comp_k$ is the opposite simplicial set, which has the same realization.} Unfortunately, the image of $r_n$ does not see enough of $w\tilde{S}_n\Var_{k,w}$ to allow for a comprehensive comparison.  However, the image factors through $w\tilde S_n\Var_{k,pw}$ (from \autoref{Varkpw}), and \edit{we will prove the following result, which states} that this induces an equivalence on realizations.

\begin{prop}\label{prop:KComp}
    The functor $r_{\sbt}$ induces an equivalence of $K$-theory spectra $K^\square(\Comp_k) \simeq K^\square(\Var_{k,pw})$.
\end{prop}

Therefore, by \autoref{Varkpw_retract}, we get the following conclusion.
\begin{cor}
    The $K$-theory spectrum $K^\square(\Comp_k)$ is a retract of $K(\Var_k)$.
\end{cor}

Before we prove \autoref{prop:KComp}, we introduce a helper object that will make the proof easier to analyze.  For any variety $A$ (not necessarily complete), let $\E_A$ be the category whose objects are pairs $(X \rcofib Y, \varphi:Y\smallsetminus X \rto A)$ of a closed immersion of complete varieties $X \rcofib Y$ and a proper piecewise isomorphism $Y \backslash X \rto A$.  A morphism
\[(\makeshort{X \rcofib Y, Y\backslash X \rto^\varphi A}) \rto (\makeshort{X' \rcofib Y', Y' \backslash X \rto^{\varphi'} A})\]
is represented by a diagram
\begin{diagram}
    {X & Y & Y\backslash X & A\\ X' & Y' & Y'\backslash X' & A \\ };
    \cofib{1-1}{1-2} \cofib{2-1}{2-2} 
    \to{1-1}{2-1} \to{1-2}{2-2} \to{1-3}{1-2}^\circ \to{2-3}{2-2}^\circ
    \to{1-3}{1-4}^\varphi \to{2-3}{2-4}^{\varphi'}
    \to{1-3}{2-3}
    \eq{1-4}{2-4}
\end{diagram}
where the square on the left is a pullback square, the map $Y\smallsetminus X \to Y'\smallsetminus X'$ is a (necessarily proper) piecewise isomorphism, and the right-hand square commutes.

The reason this category is currently of interest is the following:
\begin{lemma} \label{lem:EA}
    For any object $Y_*$ in \edit{\sout{$T_n\Var_{k,pw}$}$w\tilde S_n\Var_{k,pw}$}, the category \edit{\sout{$\E_{Y_n\smallsetminus Y_0}$}$\E_{Y_{0n}}$} is equivalent to the category $r_n/Y_*$.
\end{lemma}
\begin{proof}

We fix an object \edit{ \sout{$ Y_0\to Y_1 \to\dots \to Y_n$}}
\begin{diagram}
        {\emptyset & Y_{01} & Y_{02} & \dots & Y_{0n}\\
        & \emptyset & Y_{12} & \dots & Y_{1n} \\ & &\dots & &  \dots\\
        & & & & Y_{n-1,n} \\
        &&&& \emptyset\\};
        \to{1-1}{1-2} \to{1-2}{1-3}^\circ \to{1-3}{1-4}^\circ \to{1-4}{1-5}^\circ
        \to{2-2}{1-2} \cofib{2-3}{1-3} \cofib{2-5}{1-5}
        \to{2-2}{2-3} \to{2-3}{2-4}^\circ \to{2-4}{2-5}^\circ \cofib{3-3}{2-3} \cofib{3-5}{2-5} \cofib{4-5}{3-5}    \cofib{5-5}{4-5}    
    \end{diagram} 
or $Y_*$ for short, in \edit{\sout{$T_n\Var_{k,pw}$}$w\tilde S_n\Var_{k,pw}$}.  We define $F: r_n/Y_* \rto \E_{\edit{Y_{0n}}}$ by taking an object $$X_0 \rcofib X_1 \rcofib \dots \rcofib X_n$$ in $T^{\op}_n \Comp_k$  with \edit{\sout{a diagram of distinguished squares DIAGRAM}}a proper piecewise isomorphism $\alpha_{0n}:X_n\smallsetminus X_0 \to Y_{0n}$ and induced proper piecewise isomorphisms $\alpha_{ij}:X_{n-i}\smallsetminus X_{n-j} \to Y_{ij}$, which form an object in $r_n/Y_*$,
  to the pair
  \[(X_0 \rcofib X_n, \alpha_{0n}: X_n \backslash X_0 \rto Y_{0n}).\]
On morphisms, this simply restricts a morphism to the $0$-th and $n$-th components; the fact that it commutes above $Y_*$ implies that these satisfy the desired compatibility. 

  We define $G: \E_{Y_{0n}} \rto r_n/Y_*$
  as follows. Given an object $(P\rcofib Q, Q\smallsetminus P \xrightarrow{\varphi} Y_{0n} )$, we define $Z_i$ to be the pullback
  \begin{diagram}
      {Z_i & Q\smallsetminus P \\ Y_{0i} & Y_{0n}\\};
      \to{1-1}{1-2}^\circ \to{2-1}{2-2}^\circ 
      \to{1-1}{2-1} \to{1-2}{2-2}^\varphi
  \end{diagram}
for $0\leq i \leq n$. Note that each $Z_i$ is open in $Q\smallsetminus P$ \edit{and therefore open in $Q$, and moreover has a proper piecewise isomorphism to $Y_{0i}$}. We set $X_i  = Q \smallsetminus Z_{n-i} $, \edit{which is closed in $Q$ and in particular complete. Note that $X_{n-i}\smallsetminus X_{n-j} = (Q\smallsetminus Z_j)\smallsetminus(Q \smallsetminus Z_i) = Z_j\smallsetminus Z_i$ has a proper piecewise isomorphism to $Y_{0j}\smallsetminus Y_{0i} = Y_{ij}$.} Therefore sequence of closed embeddings of complete varieties
$$P = X_0 \rcofib X_1 \rcofib \dots \rcofib X_n = Q$$ together with \edit{\sout{the diagram DIAGRAM}the restrictions of $\varphi$ to $X_{n-j}\smallsetminus X_{n-i}$} 
  gives an object in $r_n/Y_*$. 
 
  It remains to check that these are mutually inverse equivalences. That $FG \cong 1_{\E_{Y_n\smallsetminus Y_0}}$ follows since $G$ takes an object $(X\hookrightarrow Y,\varphi)$ to an object in $T_n^{\edit{\op}}\Comp_k$ with $X$ and $Y$ as first and last object respectively, and $F$ is just restriction to these objects.
  
 That $GF \cong 1_{r_n/Y_*}$ follows since $F$ remembers $X_0\hookrightarrow X_n$ and \edit{\sout{$\varphi:X_n\smallsetminus X_0 \to Y_n$}$\varphi:X_n\smallsetminus X_0 \to Y_{0n}$}, from this we recover $X_i$ as\edit{\sout{$X_n\smallsetminus\varphi^{-1}(Y_{n-i}\smallsetminus Y_0)$}$X_n \smallsetminus \varphi^{-1}(Y_{0,n-i)}$ }.
 \end{proof}

We are now ready  to prove \autoref{prop:KComp}.
\begin{proof}[Proof of \autoref{prop:KComp}]
Directly from the definition we see that $r_n: T^\edit{\op}_n\Comp_k \rto w\tilde S_n\Var_{k,pw}$ is a symmetric monoidal functor since the symmetric monoidal structure is given by disjoint union on both sides. By \autoref{Einftycriterion}, it suffices to show that each $r_n$ is an equivalence. 

  By Quillen's Theorem A, it suffices to check that for every object $Y_*$ of $T^{\edit{\op}}_n\Var_{k,pw}$, the comma category $r_n/Y_*$ is filtered and hence contractible.  By \autoref{lem:EA} it suffices to check that $\E_A$ is cofiltered (and therefore contractible) for all $A$. 
In order to show it is cofiltered, we need to prove that it is nonempty, that for any two objects there is another object above them, and that equalizing morphisms exist.

\noindent
\textbf{$\E_A$ is nonempty:} Let $B$ be an arbitrary complete variety such that there is an open immersion $A \rto^\circ B$.  Set $C = B\smallsetminus A$, so that there is a closed immersion $C \rcofib B$.  Then $(C \rcofib B, A \req A)$ is an object in $\E_A$, as desired.

\noindent\textbf{Any two objects have another object above them:} Suppose we are given two objects
\[(\makeshort{X \rcofib Y, Y \smallsetminus X \rto A}) \qqand (\makeshort{X' \rcofib Y', Y' \smallsetminus X' \rto A})\]
in $\E_A$.  Let $A'$ be the pullback
\begin{diagram}
    { A' & Y \smallsetminus X \\ Y' \smallsetminus X' & A\\};
    \arrowsquare{}{}{}{}
\end{diagram}
then $A'\to Y \smallsetminus X $ and $A'\to Y'\smallsetminus X'$ are proper piecewise isomorphisms. Indeed, applying the pullback pasting lemma enough times, one can check that a filtration on $A$ which exhibits $Y\smallsetminus X \to A$ as a piecewise isomorphism, can be pulled back to give a filtration on $Y'\smallsetminus X'$ exhibiting $A'\to Y'\smallsetminus X'$ as a piecewise isomorphism. 

Consider the map $A' \to Y \times Y'$ induced by $A' \rto Y$ and $A' \rto Y'$.  By Nagata's compactification theorem we can factorise this as an open immersion followed by a proper map
$$A' \rcofib Z \to Y \times Y' $$
where, by replacing $Z$ with the closure of the image of $A'$ in it, we may assume $A' \rto^\circ Z$ is an isomorphism onto a dense open subvariety. In particular this gives us a complete variety $Z$ with proper $p:Z\to Y$ and $p':Z \to Y'$, induced by the projections from $Y\times Y'$ down to $Y$ and $Y'$. 

Set $W = Z\smallsetminus A'$.  It suffices to show that the proper map $Z \rto Y$ induces a map $W \rto X$ (the statement for $Y'$ and $X'$ will follow analogously), so that there is a morphism 
\begin{diagram}
    { W & Z & A' & A \\ X & Y & Y\smallsetminus X & A. \\};
    \cofib{1-1}{1-2} \cofib{2-1}{2-2} \to{1-1}{2-1} \to{1-2}{2-2}
    \to{1-3}{1-2}_\circ \to{2-3}{2-2}_\circ \eq{1-4}{2-4}
    \to{1-3}{1-4} \to{1-3}{2-3} \to{2-3}{2-4}
\end{diagram}
To prove this, it suffices to check that $W \cong Z\times_Y X$.   Indeed, since $A'$ is dense in $Z$, and the middle vertical maps above are proper, the middle square
is a pullback square by \cite[Lemma 6.15]{kuijper}, and can thus be extended to a pullback square on the left, which is then by construction a formal blowup square, as desired.

\noindent
\textbf{Equalizing morphisms exist:}  Suppose we are given $f,g:\alpha \rto \beta$ in $\E_A$.  We want to show that there exists $h:\gamma \rto \alpha$ in $\E_A$ such that $hf = hg$. Write
\[\alpha = (\makeshort{X \rcofib Y, Y\smallsetminus X \rto A}) \qqand \beta=(\makeshort{X' \rcofib Y', Y' \smallsetminus X' \rto A}),\]
with the data of $f$ given by $f_X: X \rto X'$ and $f_Y: Y \rto Y'$ (and $g$ analogously). Since both $f_Y$ and $g_Y$ make
\begin{diagram}
    { Y & Y \smallsetminus X & A\\ Y' & Y' \smallsetminus X' & A\\};
    \to{1-2}{1-1}_\circ \to{2-2}{2-1}_\circ \to{1-2}{2-2}
    \to{1-2}{1-3} \to{2-2}{2-3} \eq{1-3}{2-3}
    \diagArrow{->,xshift=-10pt}{1-1}{2-1}_{f_Y}
    \diagArrow{->,xshift=10pt}{1-1}{2-1}^{g_Y}
\end{diagram}
commute, it follows that $f_Y$ and $g_Y$ coincide on $Y \smallsetminus X$. Therefore they must also coincide on the closure $Z\defeq\overline{Y\smallsetminus X}$, which is closed in the complete variety $Y$ and therefore complete.  Set $W = X\times_Y Z$, so that there is a well-defined morphism 
\begin{diagram}
    { W & Z & Y \smallsetminus X & A\\ X & Y & Y \smallsetminus X & A.\\};
    \cofib{1-1}{1-2}  \cofib{2-1}{2-2}
    \to{1-1}{2-1}_{h_X} \to{1-2}{2-2}^{h_Y}
    \to{1-3}{1-2}_\circ \to{2-3}{2-2}_\circ \eq{1-3}{2-3}
    \to{1-3}{1-4} \to{2-3}{2-4} \eq{1-4}{2-4}
\end{diagram}
By definition, $f_Yh_Y=g_Yh_Y$, and (since all relevant squares are pullback squares) it follows that $f_Xh_X = g_Xh_X$.  This gives the desired equalizing morphism.
\end{proof}

\begin{remark}
    The step in the proof that used the properness assumption on $\Var_{k,pw}$ was the ``Any two objects have another object above them'' step.  If this step could be proved without this assumption, the above proof would show that $K(\Comp_k) \simeq K(\Var_{k,w})$.
\end{remark}

\subsection*{Smoothness considerations}

Bittner's presentation needed only smooth projective varieties, and only the blowup squares, while we have been focusing on complete varieties and formal blowup squares.  To make a construction directly analogous to Bittner's we can also consider the squares category of blowup squares.  




\begin{example} \label{ex:varieties}
  Let $\Blow_k$ be the squares category on smooth and complete varieties, generated by blowup squares. Then its distinguished object is $\emptyset$, its horizontal morphisms are closed immersions, its vertical morphisms are compositions of blowups of varieties, and its distinguished squares are horizontal and vertical compositions of blowup squares. Recall that a blowup of a variety along itself is empty.  Thus the squares
  \begin{diagram}
      { \emptyset & X \\ Y & X \amalg Y \\};
      \cofib{1-1}{1-2} \cofib{2-1}{2-2}
      \fib{1-1}{2-1} \fib{1-2}{2-2}
  \end{diagram}
  are all valid squares in the category.  Thus the category $\Blow_k$ satisfies condition ($*$) in \autoref{lem:K0norm}, making $K_0(\Blow_k)$ abelian.
\end{example}

\begin{example}
  Let $\SmComp_k$ be the ``full double subcategory'' of $\Comp_k$ which contains all smooth complete varieties.  In other words, its objects are smooth complete varieties, its horizontal morphisms are closed immersions, its vertical morphisms are proper morphisms of varieties, and the distinguished squares formal blowup squares.
\end{example}

We observe that there is an inclusion of squares categories compatible with disjoint unions
$\Blow_k \to \SmComp_k \to \Comp_k$
and therefore these induce morphisms of spectra
$$K^\square(\Blow_k) \to K^\square(\SmComp_k) \to K^\square(\Comp_k).$$

\begin{lemma}
If $k$ has characteristic $0$, there are isomorphisms
    \[K_0(\Blow_k) \xrightarrow{\sim} K_0(\SmComp_k) \xrightarrow{\sim} K_0(\Comp_k) \xrightarrow{\sim} K_0(\Var_k)\]
    each given by sending a class $[X]$ to $[X]$. 
\end{lemma}

\begin{proof}
   By Bittner's Theorem \cite{bittner04}, the composition $K_0(\Blow_k) \rto K_0(\Var_k)$ is an isomorphism. 
    The rightmost map is an isomorphism because of Nagata compactification, see e.g. \cite[Proposition 7.10]{kuijper} .  This implies that the composition $K_0(\Blow_k) \rto K_0(\SmComp_k) \rto K_0(\Comp_k)$ is an isomorphism, and therefore the map $K_0(\Blow_k) \rto K_0(\SmComp_k)$ is monic and the map $K_0(\SmComp_k) \rto K_0(\Comp_k)$ is epic.  However, $K_0(\Blow_k) \rto K_0(\SmComp_k)$ is also surjective, since it is a quotient: both groups are generated by the smooth complete varieties.  Thus it is a bijection, and therefore an isomorphism.  Consequently, the map $K_0(\SmComp_k) \rto K_0(\Comp_k)$ must also be an isomorphism, as desired. 
\end{proof}

\begin{question}[Bittner's Presentation for $K(\Var_k)$] \label{conj:bittner}
  Which of the maps
  \[\setlen{1.5em}{K^\square(\Blow_k) \to K^\square(\SmComp_k) \to K^\square(\Comp_k)\simeq K^\square(\Var_{k,pw})\to K^\square(\Var_{k,w}) \simeq K(\Var_k)}\]
  are equivalences when $k$ has characteristic $0$? What about other base fields or schemes? If not, can we identify any of the homotopy fibers?
\end{question}

\begin{remark}
    Blowup squares have one important weakness relative to abstract blowup squares: they do not \emph{compose}, in the following sense.  Usually, the composition of two blowup maps with smooth centre is not another blowup map with a smooth centre, so it is not possible to stack two blowup squares ``one on top of the other.''  (This does not contradict the well-definedness of the category with squares structure, since the inclusion of the exceptional divisor $E \rcofib \mathrm{Bl}_Y X$ cannot appear as the lower horizontal morphism in a blowup square.)  Similarly, blowup squares generally cannot compose ``next to one another.''  This means that in $N_{\sbt}T_{\sbt}\Blow_k$, most of the nontrivial structure in simplices arises from automorphisms of varieties or from the group completion structure.  It seems unlikely that this would produce all of the same complication in higher homotopy groups that it is possible to construct using abstract blowup squares or simple piecewise-automorphisms of varieties.  It therefore feels more unlikely that the first map is an equivalence than that the second map is an equivalence. Moreover, Bittner's proof that $K_0(\Blow_k)
    \cong K_0(\Var_k)$ uses the weak factorisation theorem (\cite{weakfactorization1}, \cite{weakfactorization2}) in a crucial way. Therefore any proof that $K(\Blow_k)\cong K(\Var_k)$ should use weak factorization as well. 
\end{remark}

\section{A derived motivic measure} \label{sec:derived}

\edit{\sout{The identification of $K(\Var_k)$ with}The interpretation of the Grothendieck spectrum of varieties as }$K^\square(\Var_k)$ gives rise to new possibilities for constructing derived motivic measures. \edit{\sout{In particular, any lax monoidal functor of squares categories  $\Var_k\to \mathcal{C}$ induces a map of spectra $K(\Var_k)\to K^\square(\mathcal{C})$.}Naturally, any lax symmetric monoidal functor of squares categories $\mathcal{C}\to \mathcal{D}$ induces a map of spectra
$$K\edit{^\square}(\mathcal{C}) \to K^\square(\mathcal{D}).$$
Moreover, in \autoref{lem:map_on_$K$-theory} we show that in certain cases, a functor from the ambient category of a squares category $\mathcal{C}$ into an simplicial model category  $\mathcal{A}$ can also induce a map of spectra
$$K^\square(\mathcal{C})\to K^\infty(\mathcal{E}),$$ where $\mathcal{E}$ is the $\infty$-category presented by a suitable subcategory of $\mathcal{A}$.
We will apply this to construct maps out of $K^\square(\Var_k), K^\square(\Comp_k), K^\square(\SmComp_k)$ and $K^\square(\Blow_k)$. 
\sout{Similarly, maps out of $K^\square(\SmComp_k)$ and $K^\square(\Comp_k)$ can be constructed.} }

\begin{definition}
    Let $\C$ be a category, and suppose that $\D$ is a squares category such that $\M_\D$ and $\E_\D$ are subcategories of $\C$ and such that distinguished squares in $\D$ commute in $\C$.  Then we say that $\C$ is an \emph{ambient} category for $\D$. \edit{ If $P$ is a set of commutative squares in $\C$, and $\D$ is the minimal squares category for which the objects of $\M_\D$ and $\E_\D$ are the objects of $\C$, and in which all $P$-squares are distinguished, then $D$ is the squares category \emph{generated} by $P$. }
\end{definition}

\edit{\sout{A source of such functors is the following.}}Suppose $\mathcal{D}$ is a squares category with ambient category $\C$. One can consider the distinguished squares of $\D$, or a set of commutative squares in $\C$ that generate $\D$, as a cd-structure on $\mathcal{C}$, and consider the topology generated by it. Depending on whether $\C$ has a strict initial object or not, one may want to consider the associated topology in the sense of Voevodsky \cite{voevodsky}, or a slightly coarser variant of this topology, see \cite[Definition 2.3]{kuijper}. 
If the cd-structure is \textit{nice}, meaning that it is complete and regular (or c-complete and c-regular) and compatible with a dimension function (\cite[Definition 3.2]{kuijper}), then the following proposition applies.
\begin{proposition}[{\cite[Proposition 3.9]{voevodsky} and \cite[Proposition 4.7]{kuijper}}]\label{prop:nicecdstructures}
Let $\mathcal{C}$ be category equipped with a nice cd-structure, and let $\mathcal{E}$ be a complete $\infty$-category. Then a presheaf 
$$F:\mathcal{C}^\op \to \mathcal{E}   $$
is a hypersheaf for the associated (coarse) topology if and only if it sends distinguished squares to pullback squares. 
\end{proposition}
Let $\mathcal{D}$ be a squares category with ambient category $\mathcal{C}$, and let $P$ denote a set of generating distinguished squares for the squares category $\mathcal{D}$, considered as commutative squares in $\mathcal{C}$. For what follows, in fact it suffices that $P$ is contained in a nice cd-structure $P'$ on $\C$. Let 
$$y: \mathcal{C} \to \mathrm{PSh}(\mathcal{C};\mathcal{S})$$
denote the $(\infty,1)$-Yoneda embedding. Let
$$L:\mathrm{PSh}(\mathcal{C};\mathcal{S}) \to \mathrm{HSh}(\mathcal{C};\mathcal{S})$$ denote the localization from the \edit{$\infty$-category} of presheaves of spaces to the \edit{$\infty$-category} of hypersheaves for the topology generated by $P'$. Lastly, let $$\Sigma_+^ \infty: \mathrm{HSh}(\mathcal{C};\mathcal{S}) \to \mathrm{HSh}(\mathcal{C};\mathrm{Spectra})$$
denote the stablization functor from hypersheaves of spaces to hypersheaves of spectra. As a result of  \autoref{prop:nicecdstructures},  for 
  \begin{diagram}
    { A & B \\ C & D \\}; \to{1-1}{1-2} \to{2-1}{2-2} \to{1-1}{2-1}
    \to{1-2}{2-2}
  \end{diagram}
a square in $P'$, the induced square of representable hypersheaves
 \begin{diagram}
    { Ly_A & Ly_B \\ Ly_C & Ly_D \\}; \to{1-1}{1-2} \to{2-1}{2-2} \to{1-1}{2-1}
    \to{1-2}{2-2} 
  \end{diagram}
is a pushout square in $\mathrm{HSh}(\mathcal{C};\mathcal{S})$ (see \cite[Corollary 2.16]{voevodsky}). Since $\Sigma_+^\infty$ is a left adjoint, it preserves pushout squares. Therefore the composition
$$\Sigma_+^\infty Ly:\mathcal{C} \to \mathrm{HSh}(\mathcal{C};\mathrm{Spectra}) $$
sends squares in $P'$ to bicartesian squares in the stable $\infty$-category $\mathrm{HSh}(\mathcal{C};\mathrm{Spectra})$. Since squares in $P\subseteq P'$ generate the squares category $\D$, in fact all $\mathcal{D}$-squares are sent to bicartesian squares. 
\edit{\sout{This map induces a morphism on $K$-theory spectra by the following lemma.}}

\edit{Since $\D$-squares are used to construct the spectrum $K^\square(\D)$, and bicartesian squares are used to construct the $K$-theory spectrum of the $\infty$-category $\mathrm{HSh}(\mathcal{C};\mathrm{Spectra})$, one can expect that $\Sigma^\infty_+Ly$ induces a map of $K$-theory spectra $$K^\square(\D) \to K^\infty(\mathrm{HSh}(\mathcal{C};\mathrm{Spectra})).$$ To make this precise, we use the simplicial model category  $\textup{sPSh}(\D)^{\textup{stab}}$ as defined in \cite[Section 2]{jardine}, which presents $\mathrm{HSh}(\mathcal{C};\mathrm{Spectra})$.  The following lemma gives sufficient conditions for a functor from a squares category to a Waldhausen category given by a model category to induce a functor of $K$-theory spectra.}

\begin{lemma}\label{lem:map_on_$K$-theory}
    Let $\C$ be a small 1-category with a set of distinguished squares $P$ and a distinguished object $O$, and let $\D$ be the squares category generated by these squares. Let $\mathcal{A}$ be a simplicial model category with an initial object $\emptyset$, and let $\B$ be a small full subcategory of cofibrants which admits all homotopy pushouts and is a Waldhausen category via the model structure on $\mathcal{A}$. Let 
    $$f:\C \to \B$$
    be a functor (of simplicial categories, where $\C$ is considered a simplicial category via the inclusion of $\textup{Set}$ in $\textup{sSet}$) sending squares in $P$ to homotopy pushout squares and $O$ to $\emptyset$. Then there is an induced morphism on $K$-theory spectra 
    $$K^\square(\D) \to K^\infty(N(\B^{cf}))$$
     where $K^\infty(N(\mathcal{B}^{cf}))$ is the $K$-theory of the underlying $\infty$-category $N(\mathcal{B}^{cf})$ of $\mathcal{B}$.
\end{lemma}
\begin{proof}
   Let $\mathcal{B}'$ be the squares category related to the Waldhausen category $\B$, according to Example \ref{ex:wald}. Then $f$ can be considered as a functor of squares categories
$$f:\D \to \mathcal{B}',$$
inducing a morphism of $K$-theory spectra
$$ K^\square(\D)\to K^\square(\mathcal{B}').$$
Lemma \ref{lem:wald} gives an equivalence $K^\square(\mathcal{B}')\simeq K^W(\mathcal{B})$. By \cite[Theorem 7.8]{blumberg_gepner_tabuada} there is an equivalence $K^W(\mathcal{B})\simeq K\edit{^\infty}(N(\mathcal{B}^{cf}))$, where $K\edit{^\infty}(N(\mathcal{B}^{cf}))$ is the $K$-theory of the $\infty$-category $N(\mathcal{B}^{cf})$.
\end{proof}

\edit{Let
$$\Sigma^\infty_+y:\C \to \textup{sPSh}(\D)^{\textup{stab}}$$ be the simplicial functor that sends an object in $\C$ to the simplicial presheaf of spectra represented by it. Then this functor presents the functor of $\infty$-categories $\Sigma^\infty_+Ly$. By the observations above, this functor satisfies the conditions of Lemma~\ref{lem:map_on_$K$-theory}.  Thus we obtain the desired map on $K$-theory.}


\edit{
For $\C$ any site, let $\mathrm{HSh}(\mathcal{C};\mathrm{Spectra})^f$ denote the full subcategory of finite objects in $\mathrm{HSh}(\mathcal{C};\mathrm{Spectra})$, i.e., hypersheaves $F$ such that $\hom(F,-)$ commutes with directed colimits. We note that for $\C$ a site with a topology generated by a nice cd-structure, hypersheaves of the form $\Sigma_+^\infty Ly(X)$ are finite since they are representable, and any cover in $\D$ contains a finite (simple) subcover. 
\begin{cor}\label{cor:sheaves}
    Let $\C$ be a small 1-category with a set of distinguished squares $P$ and a distinguished object $O$, and let $\D$ be the squares category generated by $P$.  Assume that $P$ is included in a nice cd-structure $P'$ on $\C$. Then the Yoneda embedding induces a map of $K$-theory spectra
    $$K^\square(\D) \to K^\infty(\mathrm{HSh}(\mathcal{C};\mathrm{Spectra})^f).$$
\end{cor}
\begin{proof}
   We note that the functor of simplicial categories
    $$\Sigma^\infty_+y:\C \to \textup{sPSh}(\C)^{\textup{stab}} $$ lands in cofibrant objects. Let $\B$ denote the subcategory of $\textup{sPSh}(\C)^{\textup{stab}}$ generated by the initial presheaf and the image of $\Sigma^\infty_+y$ under homotopy pushouts. Then we can apply \autoref{lem:map_on_$K$-theory} to
    $$\Sigma^\infty_+y:\C\to \B$$
    obtain a map 
    $$K^\square(\D) \to K^\infty(N(\B^{cf})).$$
    Now we observe that $N(\B^{cf})$ embeds into $\mathrm{HSh}(\mathcal{C};\mathrm{Spectra})^f$, since $\B$ is generated by representable (hence finite) hypersheaves under pushouts. This embedding preserves pushouts, so there is a map on the $S_\infty$-construction that gives a map
$$K^\infty(N(\C^{cf})) \to K^\infty(\mathrm{HSh}(\mathcal{C};\mathrm{Spectra})^f).$$
\end{proof}

}

In the case of the squares categories $\Blow_k$, $\SmComp$ and $\Comp_k$, the generating squares (blowup squares of smooth and complete varieties, abstract blowup squares of smooth and complete varieties, and abstract blowup squares of complete varieties, respectively) form nice cd-structures, denoted $B$, $ASC$ and $AC$ respectively \cite[Proposition 6.8]{kuijper}. We consider the associated topologies $\tau_B$, $\tau_{ASC}$ and $\tau_{AC}$ on the categories $\Blow_k$, $\SmComp$ and $\Comp_k$. 

In the case of the squares category $\Var_k$, to our knowledge the distinguished squares can not be extended to a nice cd-structure on the category of varieties. However, consider the following category.
 \begin{definition}
    Let $\mathbf{Span}_k$ be the category with objects varieties over $k$, and as morphisms spans
$$X \lto^\circ U \rto Y $$
where $U\rto^\circ X$ is an open immersion and $U\rto Y$ a proper map. The composition of two spans $X\lto^\circ U \rto Y $ and $Y\lto^\circ V \rto Z $ is defined by taking the pullback of $U \to Y$ along $V \rto^\circ Y$.
\end{definition}
We now define a squares category with $\Span_k$ as ambient category. 
\begin{definition}
    Let $\Span_k$ denote the squares category with $\Span_k$ as ambient 1-category, maps of the form 
    $$X   \req X \rcofib^f Y $$ with $f$ a closed immersion as horizontal morphisms, and maps of the form 
    $$X \lto^\circ U \req U $$
    as vertical morphisms. The distinguished squares are squares of the form 
    \begin{diagram}
        {C & X \\ U \cap C & U \\};
        \to{1-1}{1-2} \to{2-1}{2-2} \to{2-1}{1-1}^\circ \to{2-2}{1-2}^\circ 
        \end{diagram}
        where $X  = U \cup C$.
\end{definition}

We note that this is morally the same squares category as $\Var_k$, only the orientation of the squares is different. In particular, $K^\square(\Span_k)$ is isomorphic to $K^\square(\Var_k)$, this can be seen by comparing the simplicial categories $T_{\sbt} \Span_k$ and $T_{\sbt} \Var_k$ . 

\begin{remark} $\hbox{ }$
\begin{enumerate}
    \item    The category $\mathbf{Span}_k^\op$ is in fact the category associated to the squares category $\Var_k$ according to Definition 2.5 in \cite{fiedorowicz_loday}; equivalently, $\Span_k$ is the category associated to the transpose squares category $\Var_k^t$. 
    \item  In $\Span_k$ the distinguished object $\emptyset$ is now terminal in the vertical arrows, so strictly speaking it is not a squares category anymore. However, with appropriate adjustments to the proofs, all the results about squares categories still hold. In particular we can construct the $K$-theory spectrum $K^\square(\Span_k)$ in the same way, and the presentation of $K_0(\Span_k)$ given by \autoref{lem:K0norm} still holds. 
\end{enumerate}
\end{remark}
On $\mathbf{Span}_k$ we consider the coarse topology $\tau^c_{A\cup L}$, generated by the cd-structure $A\cup L$. Here $A$ is the set of abstract blowup squares (of arbitrary, possibly non-complete varieties) and $L$ the set of squares of the form
\begin{diagram}
    {X\smallsetminus U & X \\ \emptyset & U. \\ };
\cofib{1-1}{1-2} \cofib{2-1}{2-2} \to{2-1}{1-1}^\circ \to{2-2}{1-2}_\circ    
\end{diagram}
Then $A\cup L$ is a nice cd-structure by \cite[Proposition 6.8]{kuijper}. 

\begin{rmk}
    The category $\mathbf{Span}_k$ is the natural domain for ``compactly supported cohomology theories" of algebraic varieties, since these are contravariant in proper morphisms and covariant in open immersions. For $\C$ a stable $\infty$-category with a t-structure and 
    $$F:\mathbf{Span}^\op \to \C$$
    a presheaf, such as the singular cochain functor, we can define a compactly supported cohomology theory by setting $H^{n}_c(X):= \pi_{-n}F(X)$. If $F$ is in the category $\mathrm{HSh}(\mathbf{Span}_k;\C)_\emptyset$ of hypersheaves for $\tau^c_{A\cup L}$ satisfying $F(\emptyset)=*$, then the sheaf condition implies the existence of a long exact sequence 
    $$\dots \to H_c^n(U) \to H_c^n(X) \to H_c^n(X\smallsetminus U) \to H^{n+1}_c(U)\to \dots $$
that is characteristic for compactly supported cohomology theories. Thus, we can view $$\mathrm{HSh}(\mathbf{Span}_k;\mathrm{Spectra})_\emptyset$$ as a category of compactly supported cohomology theories on varieties.
\end{rmk}

For $\D$ any site, let $\mathrm{HSh}(\mathcal{D};\mathrm{Spectra})^f$ denote the full subcategory of finite objects in $\mathrm{HSh}(\mathcal{D};\mathrm{Spectra})$, i.e., hypersheaves $F$ such that $\hom(F,-)$ commutes with directed colimits. We note that for $\D$ any of the categories $\Blow_k$, $\SmComp_k$, $\Comp_k$ or $\mathbf{Span}_k$, with the topology $\tau_B$, $\tau_{ASC}$, $\tau_{AC}$ or $\tau^c_{A\cup L}$, hypersheaves of the form $\Sigma_+^\infty Ly(X)$ are finite since they are representable, and any cover in $\D$ contains a finite (simple) subcover. 

\begin{prop}\label{prop:derived_motivic_measure}
    Let $\C$ be either of the categories $\Blow_k$, $\SmComp_k$, $\Comp_k$ or $\mathbf{Span}_k$, and let $\D$ be the associated squares category\edit{.\sout{$\Blow_k$, $\SmComp_k$, $\Comp_k$ or $\Span_k$.}} Then the functor
    $$\Sigma_+^\infty Ly:\mathcal{C} \to \mathrm{HSh}(\mathcal{C};\mathrm{Spectra}) $$
    induces a map of $K$-theory spectra
    $$\mu_{\D}:K^\square(\D) \to K^\infty(\mathrm{HSh}(\mathcal{C};\mathrm{Spectra})^f).$$
\end{prop}
\begin{proof}
\edit{
    For $\C$ either of the categories $\Blow_k$, $\SmComp_k$ or $\Comp_k$, \autoref{cor:sheaves} applies directly and gives the desired map of spectra. For $\mathbf{Span}_k$, we observe that a $\mathrm{Spectra}$-valued hypersheaf for this topology sends not only the squares in $A\cup L$, but also distinguished squares as in Example~\ref{ex:varieties}(3) to a cartesian squares, since $\mathrm{Spectra}$ is stable. Therefore, for such a distinguished square, the associated square of representable hypersheaves is cocartesian. Hence \autoref{lem:map_on_$K$-theory} and the arguments in the proof of \autoref{cor:sheaves} can be used to construct a map
$$\mu_{\Span_k}:K^\square(\mathbf{Span}_k) \to K^\infty(\mathrm{HSh}(\mathbf{Span}_k;\mathrm{Spectra})^f)$$
as well.}
\end{proof}

We denote by $\mathrm{HSh}(\Span_k;\mathrm{Spectra})_\emptyset^f$ the full subcategory of finite hypersheaves that send $\emptyset$ to the sphere spectrum. By mild abuse of notation, we denote by $\mu_{\Span_k}$ also the map 
$$ K^\square(\Span_k) \to K(\mathrm{HSh}(\mathbf{Span}_k; \mathrm{Spectra})_\emptyset^f)$$
induced by $\Sigma_+^\infty Ly$.

Combined with the results in Section \ref{sect:bittner} we get the following picture:
\begin{diagram}[4em]
    {K(\Var_k)  & K^\square(\Var_k) = K^\square(\Span_k) & & K(\mathrm{HSh}(\mathbf{Span}_k; \mathrm{Spectra})_\emptyset^f) \\ & K^\square(\Comp) & & K(\mathrm{HSh}(\Comp_k;\mathrm{Spectra})^f) \\ & K^\square(\SmComp_k) & & K(\mathrm{HSh}(\SmComp_k;\mathrm{Spectra})^f)\\
    & K^\square(\Blow_k) & & K(\mathrm{HSh}(\Blow_k;\mathrm{Spectra})^f)
    \\} ; 
   \to{2-2}{1-2}^{\mathrm{Prop}\ \ref{prop:KComp}}
    \we{1-1}{1-2}_{\mathrm{Lem\ }\ref{prop:oneKvar}}  \to{3-2}{2-2} \to{1-2}{1-4}^{\mu_{\Span_k}} \to{2-2}{2-4}^{\mu_{\Comp_k}} \to{3-2}{3-4}^{\mu_{\SmComp_k}}  \to{4-2}{4-4}^{\mu_{\Blow_k}} \to{4-2}{3-2} \eq{1-4}{2-4} \eq{2-4}{3-4} \eq{3-4}{4-4}
    \labar{1-2}{2-4}{\not\circlearrowleft}
    
\end{diagram}
where the equalities on the right hold because of \cite[Theorem 7.2]{kuijper}. Note that the lower two squares commute, whereas the upper square might not commute.

\bibliographystyle{alpha}
\bibliography{CZ}

\end{document}